\documentclass[11pt]{article}
\usepackage{latexsym}
\usepackage{amsmath,amsthm}
\usepackage{amssymb,amsfonts}
\usepackage{graphicx}

\def\e{\epsilon}
\def\R{{ I\!\!R}}

\def\N{\mathbb{N}}
\def\II{{\rm I\kern-0.5exI}}
\def\III{{\rm I\kern-0.5exI\kern-0.5exI}}

\numberwithin{equation}{section}
\newtheorem{theorem}{Theorem}[section]
\newtheorem{lemma}[theorem]{Lemma}

\newtheorem{proposition}[theorem]{Proposition}
\newtheorem{corollary}[theorem]{Corollary}

\begin{document}

\title
{\Large{Regularity and asymptotic behavior of laminar flames in higher dimensions}}
\date{}
\author{
 J. Rubin Abrams\footnote{Department of Mathematics, University of Arizona,
Tucson, AZ 85721. }
\and
Sunhi Choi\footnote{Department of Mathematics, University of Arizona,
Tucson, AZ 85721. }
} 
\maketitle

\begin{abstract} In this paper we study a parabolic free boundary problem,
arising from a model for the propagation of equi-diffusional premixed flames with high activation energy. Consider
the heat equation
$$ \Delta u = u_t, \,\,\, u>0$$
in an unknown domain $\Omega \subset \R^n \times (0,\infty)$ with the following boundary conditions
$$u=0,\,\,\,\,\, |\nabla u|=1$$
on the lateral  free boundary. If the initial
data $u_0$ is compactly supported, then the solution vanishes in a
finite time $T$, which is called the extinction time.  In this
paper, we prove regularity properties of the free boundary after some positive time,  and we investigate the asymptotic behavior of a solution near its
extinction time,  under certain assumptions on the initial data. We give a quantitative estimate on the flatness of the free boundary after $t = 3T/4$, and prove that the free boundary is asymptotically spherical and   the solution is asymptotically   self-similar. We also obtain that the free boundary is a graph of $C^{1+\gamma, \gamma}$ function after some positive time.
\end{abstract}

\section{Introduction}
For a continuous and nonnegative initial function 
$u_0$ defined in
$\R^n$ with a nonempty positive set, we find a
nonnegative continuous function $u$ in $\R^n \times (0, \infty)$ such
that
$$
\left\{\begin{array} {lll}
u_t=\Delta u  &\hbox{ in } & \{u >0\} \\ \\
|\nabla u|= 1, \,\,u=0 &\hbox{ on } & \partial\{u>0\} \\ \\
u(x,0) = u_0(x)
\end{array}\right.\leqno (P)
$$
where $\nabla u$ denotes the spatial gradient of $u$ and $\{u>0\}$
denotes the inverse image $\{(x,t): u(x,t)>0\}$.

If the initial data $u_0$ is compactly supported, then the solution vanishes in a
finite time $T$, called the extinction time. In combustion
theory for laminar flames, $u $ denotes the {\it minus temperature}
$\lambda(T_c-T)$ where $T_c$ is the flame temperature and $\lambda$
is a normalization factor (see [BL]). The region $\Omega:=\{u>0\}$
represents the {\it unburnt zone} and the lateral free boundary
$\Gamma$ of $\Omega$ represents the {\it flame front} (See [BL], [CV]  and [V] for the details in combustion theory).

\subsection{Background}

Assuming $u_0$ is bounded and Lipschitz continuous, a global weak
solution of ($P$) has been obtained by Caffarelli and V\'{a}zquez [CV]  as an asymptotic limit
of the following approximation problems
$$
\left\{\begin{array} {lll}
\partial_t u_\e=\Delta u_\e +\dfrac{1}{\e}\beta(\dfrac{u_\e}{\e})   \\ \\
u_\e(x,0)=u_{0 \e}(x)
\end{array}\right.\leqno (P_\e)
$$
where $\beta $ is a nonnegative smooth function supported on $[0,1]$
with $\int_0^1 \beta =1/2$ and $u_{0 \e}$ approximates $u_0$ in a
proper way. The family of solutions $\{u_\e\}$ is uniformly bounded
in $C^{1,1/2}_{x,t}$-norm on compact sets and they converges along
subsequences to a function $u \in C^{1, 1/2}_{loc}$, which is called
a {\it limit solution} of ($P$). (See Theorem 7.1 of [CV] for the existence, and [CV], [CLW1], [CLW2] and [W] for the properties of limit solutions.)  Uniqueness and regularity properties have been investigated by various authors ([GHV], [K], [LVW] and [P]) for limit solutions as well as {\it viscosity solutions}.  A viscosity solution, introduced in [CLW2] and [LW], is a  weak  solutions which essentially satisfies the comparison principle with the classical sub and super solutions. 

If a limit solution $u$ has a shrinking support, i.e., if $u_0$ is a $C^2$-function with
\begin{equation} \label{shrink}
\Delta u_0 \leq 0 \hbox{ in } \{u_0>0\}\,\,\hbox{ and  }\,\, 
|\nabla u_0| \leq 1 \hbox{ on }
\partial \{u_0>0\},
\end{equation}
then it was proved by Kim [K] that a limit
solution is unique and coincides with a  viscosity solution. In this paper we adopt the notion of a limit solution, and consider a natural situation (\ref{shrink}) for application in which $u$ has a shrinking support initially, i.e.,  the flame advances at the initial time.

For more general initial data $u_0 \in
C^{0,1}(\R^n)$, Weiss showed in [W] that each limit of $u_\e$ is a
solution of $(P)$ in the sense of domain variation. Given a   domain variation solution $(u, \chi)$, Andersson and Weiss proved in [AW]
that one-sided flatness of the free boundary implies regularity. In
particular, the free boundary can be decomposed  into an open
regular set with H\"{o}lder continuous space normal, and a closed
singular set. (See Lemma~\ref{AW} below.)

If an initial data is compactly supported, a solution vanishes in a finite time $T$: the time when the
unburnt zone collapses in a combustion model. Particularly, it was
proved by Galaktionov, Hulshof and V\'{a}zquez [GHV] that if the initial data is radially symmetric and
supported in a ball,  then the solution is  asymptotically {\it
self-similar} near the extinction time $T$. (See Lemma~\ref{pro11} for the definition and existence of   self-similar solutions.) Furthermore,  the authors prove that  the profile of any non-radial solution is asymptotically  self-similar, when the spatial dimension $n=1$.  In this paper, we extend their results on non-radial solutions to higher dimensions, proving that  the free boundary is asymptotically spherical and   the solution is asymptotically   self-similar under  certain assumptions on the initial data.

\subsection{ Main difficulties}
In higher
dimensions $n \geq 1$, not much is known for the behavior of non-radial
solutions of ($P$). Compactness arguments do not apply without a priori knowledge  on regularity properties. In  general setting, it is expected that topological changes of the domain (the positive set of the solution) might occur, possibly generating multiple radial profiles for later times. Due to possible topological changes, the  flatness of the boundary at time $t$ is introduced as $1- r^{in}(t)/r^{out}(t)$, where $r^{in}(t)$ and $r^{out}(t)$ are the maximal and  minimal radii of concentric spheres, which touch the boundary of the domain inside and outside, respectively. (See Figure 1.) 
\begin{figure}
\centering
\includegraphics[width=0.7\textwidth]{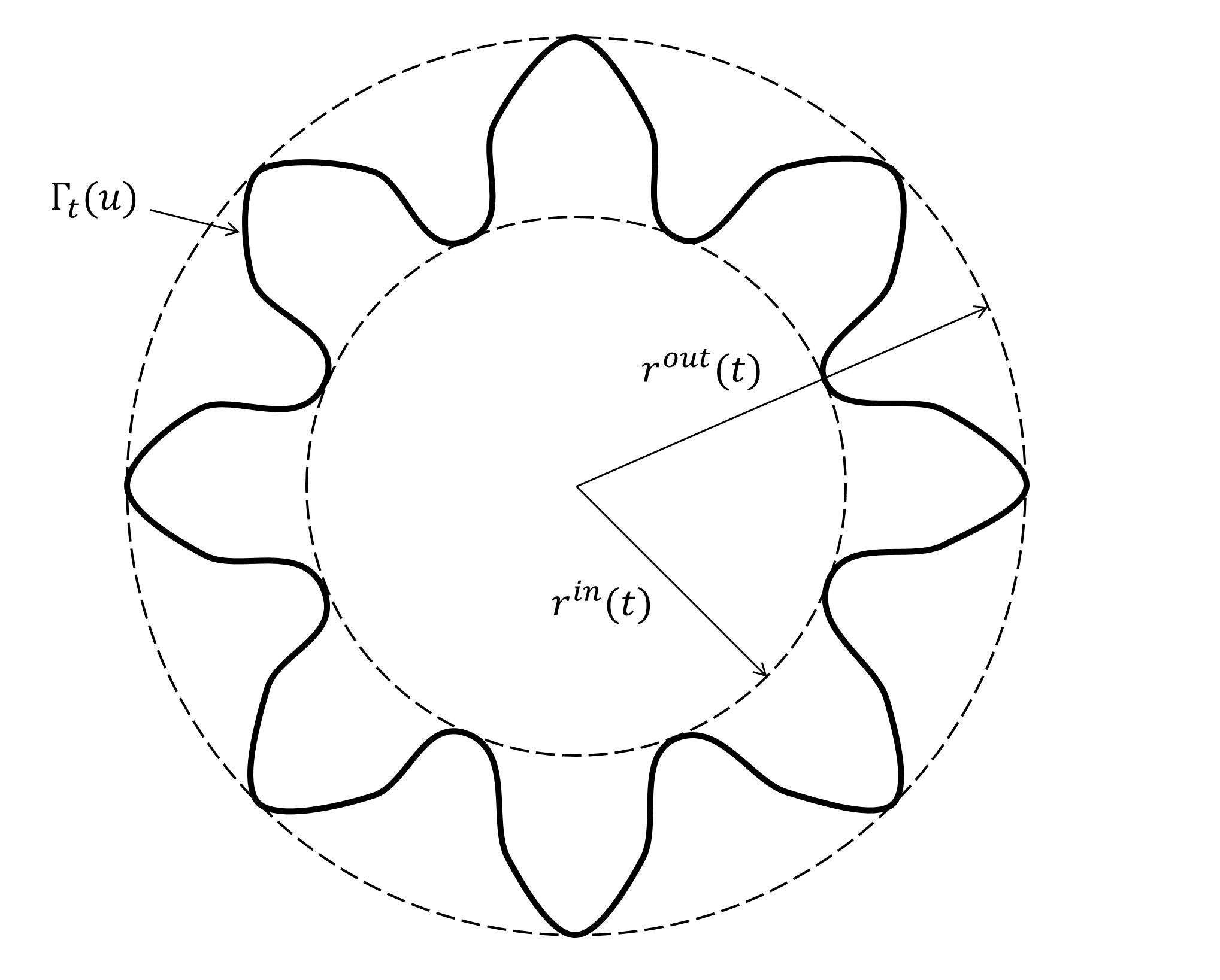} \caption{$r^{in}(t)$ and $r^{out}(t)$} 
\end{figure}
Here, all the (possibly multiple) boundary pieces  are trapped between the boundaries of these concentric spheres.  The main difficulty of this paper lies in obtaining rescaled flatness of the boundary and its  improvement  in  time,  as $t$ approaches the extinction time $T$.

Since there is no stability result on the extinction time
 $T$, i.e.,  a solution focuses at a point with a divergent boundary
 speed of $\sqrt{T-t}$ (See Lemma~\ref{lemma1}), we decompose the time interval $(0, T)$ into dyadic intervals, and rescale the solution. Under certain assumptions on the initial data (see Theorem~\ref{thm1}), we prove
 $$
r^{in}(t)/r^{out}(t) \rightarrow 1 \hbox{ as }t \rightarrow T.
$$
In fact, we obtain (see (i) of Theorem~\ref{thm1}) $$\ln(1-r^{in}(t)/r^{out}(t)) \approx \ln(T-t)$$
by   tracking the interplay between the shape of the boundary and the value of the solution in the interior: improved flatness of the boundary makes the interior value of a solution closer to a radial profile at later times (on the succeeding  dyadic time interval), and then the interior improvement propagates to the boundary, making it  flatter in the next dyadic time interval. This analysis is conducted by  rather delicate constructions of  sub- and supersolutions. 

\subsection{Open problems}

In the main theorem of this paper, flatness of the boundary is obtained when the initial data  is a sum of a radial function and a compactly supported  function $\rho$  such that $\rho$ is periodic in angle with a small $L^\infty $-norm  and a small period. (See Theorem~\ref{thm1} for precise assumptions.)  Here,  $\rho$ is assumed to be  periodic in angle for a technical reason. Due to this assumption, the free boundary is trapped between two concentric spheres with a ``fixed'' center $0$ for all time $t \in  (0,T)$. If $\|\rho\|_\infty$ is small but $\rho$ is not   periodic, then we expect that the free boundary is trapped  between concentric spheres with  ``moving'' centers in time. Hence the center of the concentric spheres, which measure  the flatness of the boundary, should be modified  for each dyadic time intervals in some appropriate way.   

\vspace{0.1 in}

 If $\rho$ is periodic but  $\|\rho\|_\infty$ is not small, then the domain may split into a number of ``large'' pieces, creating multiple profiles of solutions with similar sizes. Consider an example with an initial domain $B_1(0) \subset \R^2$,
 that large initial data around $(\pm 1/2,0)$ leads to the split of
 the domain into two large pieces at a later time. Then the solution is no longer periodic in each of those large pieces of the domain, and it focuses at multiple points other than the origin.  However,  if one assumes that $\rho$ creates only small pieces of the positive set around the main piece of the domain, i.e., the solution extincts at the origin,  then it may happen that after some positive time $\tau <T$, $u(\cdot, \tau)$ is decomposed into the sum of a large radial function and  a small non-radial function. Then the main theorem of this paper can be applied for $t \geq \tau$.    

\vspace{0.1 in}

 Lastly, the periodicity of $\rho$ is assumed to be small in the main theorem. This assumption is used in three places, where the small periodicity of $\rho$ ensures the small oscillation of the  solution $u(x,t)$ on the boundary of some inner sphere. More precisely, using the small periodicity of $\rho$,  we estimate the oscillation of $u(x,t)$  on $\partial B_{r^{in}(t)}(0)$, where $B_{r^{in}(t)}(0)$ is the maximal sphere contained in the positive set of $u(x,t)$.  Hence  the main theorem  will hold for any period $<2 \pi$,  if one can show the following:  after some time $t \geq \tau$, $$ \frac{\hbox{the oscillation of }   u(\cdot, t) \hbox{ on } \partial B_{r^{in}(t)}(0)} {r^{in}(t)} \lesssim \|\rho\|_\infty.$$  The lack of geometric assumptions on the initial free boundary makes this problem more challenging.  (The initial free boundary is not assumed to be a small perturbation of a sphere, it is located between  two concentric  spheres with ``just comparable'' radii. (See Figure 2.)) 
  \begin{figure}
   \centering
   \includegraphics[width=0.7\textwidth]{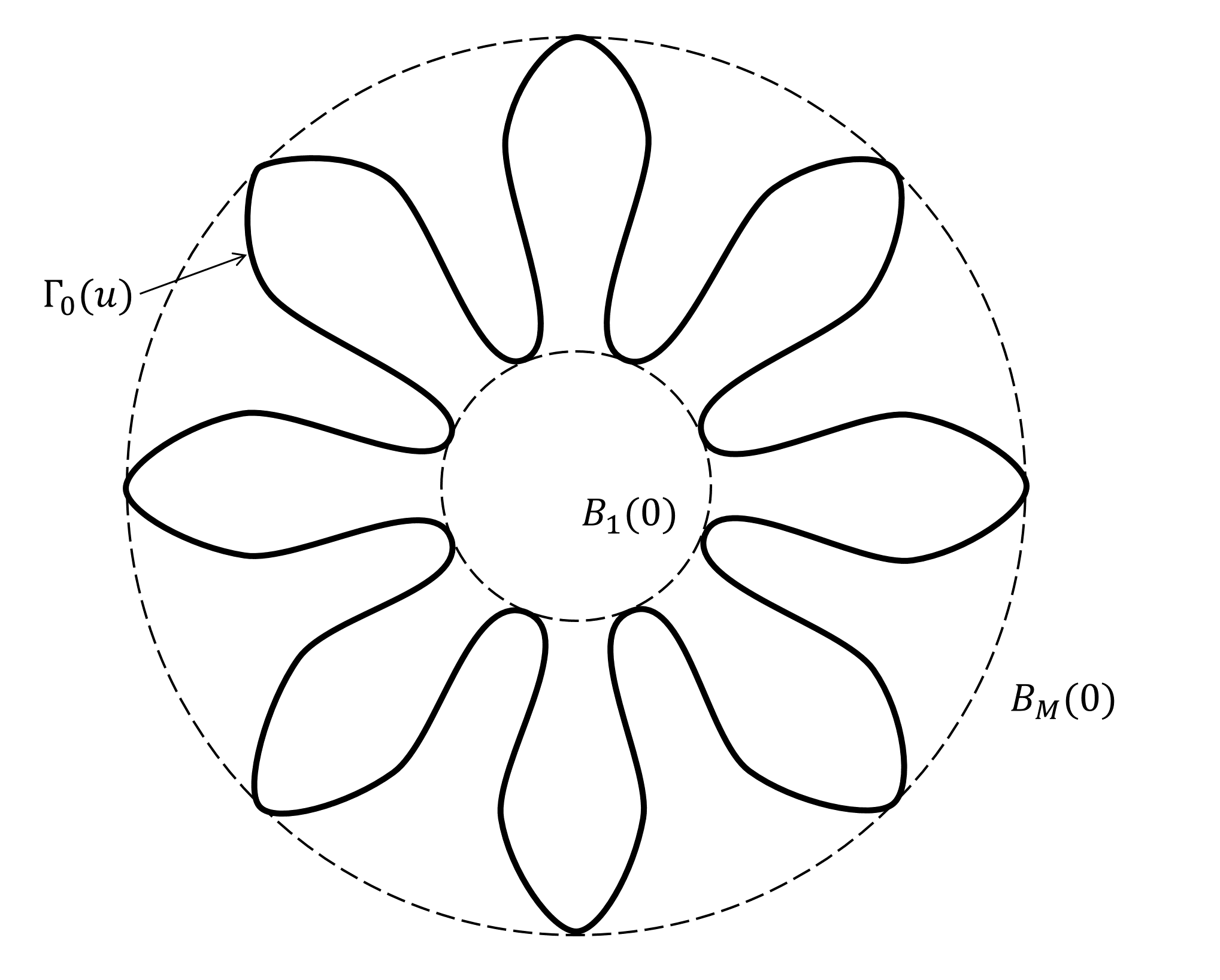}
   \caption{initial free boundary of $u$}
   \end{figure}
  If the shape of the initial boundary is much different from a sphere, the solution might lose the property of ``small oscillation''  for some period of time. More precisely, the oscillation of $u(x,t) /r^{in}(t)$ on  $\partial B_{r^{in}(t)}(0)$ might get smaller than $\|\rho\|_\infty$  possibly {\it after } some {\it waiting time}: one can consider the case that the free boundary is not trapped between closely-located concentric spheres at time $t=T/2$, creating large non-radial data at time $t=T/2$ such that 
  $$\frac{\|u(x,T/2) -\phi(x)\|_\infty} {\|u (x,T/2)\|_\infty} \gg \|\rho\|_\infty$$ where $\phi(x)$ is the maximal radial function $\leq u(x,T/2)$.

\subsection{ Main results}

Below we state the main theorem of this paper, on  the asymptotic behavior of
non-radial, space-periodic solutions  in higher dimensions $n \geq
1$. Under some conditions on the initial data,  the free boundary will be proved asymptotically spherical
near the extinction point $(0, T) \in \R^{n+1}$. Then it will turn out that the solution is
asymptotically self-similar and the free boundary is a graph of
$C^{1+\gamma, \gamma}$ function after some positive time.

For rescaling of the solution, we use the dyadic decomposition of the time interval $(0,T)$, which is  denoted by 
$0 < t_1<t_2 <...<T$.  More precisely, 
  $t_i =(1-2^{-i})T$ for $i \in \N$, i.e.,
$$ t_1 =T/2  \,\, \hbox{,  } \,\, t_{i+1}=t_i + \frac{T-t_i}{2} .$$ The free boundary of $u$ at time $t$ is denoted by $$\Gamma_t(u):= \partial \{x: u(x, t)>0\},$$ and the free boundary of $u$ in $\R^n \times (0, \infty)$  is denoted by $$\Gamma(u) := \partial \{(x,t): u(x,t)>0\} \cap \{t>0\}.$$

\begin{theorem} \label{thm1}
Let $u$ be a solution of ($P$)   with an initial data $u_0$
satisfying (\ref{shrink}).  Suppose
$$u_0=\phi_0+\rho$$
 where $\phi_0$ is a
nonnegative radial function supported on $B_1(0)$ and $\rho$ is a 
nonnegative function such that $u_0$ is compactly supported and its positive set is simply connected. Let $M>1$ be a constant satisfying  
\begin{itemize}
\item[(a)] $\{u_0>0\} \subset B_M(0)$
\item[(b)] $|\nabla u_0| \leq M $
\item[(c)] $\max \phi_0 \geq 1/M$.
\end{itemize}
Then there is a constant $\alpha(n, M)>0$
depending only on the dimension $n$ and $M$ such that the followings hold: if
$\|\rho\|_\infty  \leq \alpha$  and $\rho$ is periodic in angle  with
period $\leq \alpha$ for some  $\alpha \leq \alpha(n,M)$, then
\begin{itemize}
\item[(i)] the free boundary of $u$ is asymptotically spherical near its focusing point $0 \in \R^n$. More precisely, there exist constants $0<h<1$ and $C>0$  depending only
on $n$ and $M$ such that for $t \in [t_k , t_{k+1}]$
and $k\geq 2$,
$$
\Gamma_t(u) \subset B_{(1+C h^k\alpha) r(t)}(0) -B_{r(t)}(0) $$
where $r(t)$ is a decreasing function of $t$ with $r(T)=0$.

\item[(ii)] $u$ is approximated by radial solutions  $v_k$ of (P). More precisely, for $k\geq 2$, there is a radial solution $v_k$ of (P)  such that 
\begin{equation} \label{ss}
|u - v_k| \leq C h^{k} \sqrt{T-t} \hbox{ for } t_k \leq t < T
\end{equation}
 where $C>0$ and $0<h<1$ are constants depending only on $n$ and $M$. Then it turns out that u is  asymptotically self-similar by (\ref{ss}) and the self-similarity of radial solutions.

\item[(iii)] the free boundary $\Gamma(u) \cap \{t_{k_0} < t< T\} $ is a graph of $C^{1+\gamma, \gamma}$ function, where $k_0 \in \N$ is a  constant depending  only on $n$ and $M$. More precisely, $k_0 \geq 2$ is a constant satisfying 
$h^{k_0} < c_0 /\alpha$
where $0<h<1$ and $c_0>0$ are constants depending only on $n$ and $M$.

\end{itemize}
\end{theorem}

Below we make several remarks on the shape of the free boundary, and some of the assumptions of Theorem~\ref{thm1}.  

\vspace{0.05 in}

\noindent {\bf Remark 1.} The initial free boundary $\Gamma_0(u):=\partial \{u_0>0\}$
is not assumed
to be a slight perturbation of a sphere. It can be any irregular subset of
the annulus $B_M(0) \setminus B_1(0)$, which is periodic in angle (not
necessarily star-shaped). (See Figure 2.)

\vspace{0.05in}

\noindent {\bf Remark 2.}
Assuming  $\|u_0\|_\infty < \infty$, it
was proved in [CV] that 
$|\nabla u(x,t)| \leq  M$ for $t>T/2$ and a constant $M>0$ depending on $n$ and $\|u_0\|_\infty$. Hence without loss of generality, we suppose from the beginning that $u_0$
has bounded interior gradient, i.e., $|\nabla u_0| \leq M.$ Also for
simplicity, we assume that $\{u_0>0\}$ is simply connected. This
assumption is used only in the proof of Lemma~\ref{lemma1}.  

\vspace{0.05 in}

\noindent {\bf Remark 3.} The smallness of the period of $\rho$ can be replaced by 
$\|\nabla \rho\|_\infty \leq \alpha$.

\vspace{0.05 in}

\noindent {\bf Remark 4.}
 Since we do not assume a 
lower bound on $|\nabla \phi_0|$, even a small
function $\rho$ can change the geometry of the boundary in a significant way for a small time: even when we start from a radially symmetric  initial boundary $\Gamma_0(u)=\partial B_1(0)$ (this is the case $\{\rho >0\} \subset \{\phi_0 >0\}$),  a small function
$\rho$ can change the geometry of the boundary  near the initial time so that $r^{in}(t)  \ll r^{out}(t)$.   In fact, for a positive time  $t<T/2$, the boundary is located between two concentric spheres with
``just'' comparable radii (Lemma~\ref{lemma1}). (See Figure 3.)

\vspace{0.05 in}

\noindent {\bf Remark 5.}
Even a  small $\rho$  can change the topology of the
domain creating small pieces of the positive set around  the main
piece of the domain.  However the topological change of the domain does not affect the geometry of the main piece, under the assumptions of Theorem~\ref{thm1}.

\vspace{0.05 in}
  
  The main theorem of this paper is different from a standard
   nonlinear stability analysis, and it is not a perturbation method.   We refer to [BHS] and [BHL] for a linearized stability analysis.

\subsection{Outline of the paper}
We introduce  some preliminary lemmas
and notations in section~\ref{sec2}, and  we prove the main theorem in sections~\ref{sec3} to \ref{sec8}. First in section~\ref{sec3}, we show that the
free boundary $\Gamma_t(u)$ is located between two concentric
spheres with comparable radii at each time $t$, and prove that   the
maximal radial subregion $\Omega^{in} $ of $\Omega$ has a boundary
close to a Lipschitz graph in a parabolic scaling. Then in
section~\ref{sec4} the scaled $\alpha$-flatness of the free boundary
is obtained  when the function $\rho$ has size $\alpha$. In other
words, $\Gamma_t(u)$ is located between two concentric spheres with
the outer radius bounded by $(1+C\alpha) \times $ inner radius at
later times $t$. In section~\ref{sec5} the solution $u$ is
approximated by a radial function at interior points away from the
boundary, and this interior estimate  is improved in
section~\ref{sec6} thanks to the $\alpha$-flatness of the free
boundary. Then in section~\ref{sec7}, the interior improvement
(obtained in section 6) propagates to the free boundary at later
times giving an improved estimate on the location of the free
boundary. More precisely, if the free
 boundary is located near a sphere at each time $t \in (t_k,T)$
 then  several barrier functions will show that the solution $u$ gets
 closer to a radial function $\phi$ at later times $t \in (t_{k+1}, T)$, at points away from the
boundary (Lemma~\ref{lemma5}).  This improved estimate on the values
of $u$ forces the free boundary to be located in a smaller
neighborhood of a sphere at $t \in (t_{k+2}, T)$
(Lemma~\ref{lemma6}). By iteration, it will follow that the free
boundary is asymptotically spherical near the focusing point. In the
last section, the asymptotic behavior of the solution is
investigated and the regularity of the free boundary follows as a
corollary from the flatness of the free boundary and the radial
approximations of the solution.

\section{Preliminary lemmas and notations} \label{sec2}

Below we introduce some notations.

\vspace{0.1in}

\noindent $\bullet$ Denote by $\Omega(u)$, the positive set of $u$,
i.e.,
$$\Omega(u)= \{u>0\}=\{(x,t): u(x,t)>0\}.$$

 \noindent $\bullet$ Denote by $\Gamma(u)$, the  free
boundary  of $u$, i.e.,
$$\Gamma(u) = \partial \Omega(u) \cap \{t>0\}.$$

\noindent $\bullet$ Denote by $\Sigma_t$, the time cross section of
 a space time region $\Sigma$,   i.e.,
$$\Sigma_t=\{x: (x,t) \in \Sigma \}.$$
In particular, $$\Omega_t(u)=\{x: u(x,t)>0\}$$ and $$\Gamma_t(u) =
\partial \Omega_t(u)$$ where $\Gamma_t(u)$ is called the free boundary  of $u$ at time
$t$.

\vspace{0.1in}

\noindent $\bullet$ Denote by $B_r(x)$, the space ball with radius
 $r$, centered at $x$.

\vspace{0.1in}

\noindent $\bullet$ Denote by $Q_r(x,t)$,  the parabolic cube with
radius $r$, centered at $(x,t)$.  Denote by $Q_r^-(x,t)$, its
negative part, i.e.,
$$Q_r(x,t)=B_r(x) \times( t-r^2, t+ r^2), \quad Q_r^-(x,t)=B_r(x) \times( t-r^2, t).$$

\noindent $\bullet$ Denote by $K_r(x,t)$,  the hyperbolic cube with
radius $r$, centered at $(x,t)$, i.e.,
$$K_r(x,t)=B_r(x) \times( t-r, t+ r).$$

\noindent $\bullet$ A space time region $\Omega $ is Lipschitz  in
$Q_r(0)$ (in parabolic scaling) if
$$
Q_r(0) \cap \Omega = Q_r(0) \cap\{(x,t): x_n
> f(x',t)\}
$$
where $x=(x',x_n)\in \R^{n-1} \times \R$ and  $f$ satisfies
$$|f(x',t)-f(y',s)| \leq L (|x'-y'|+|t-s|/r)$$ for some $L>0$, with
$f(0,0)=0$.

\vspace{0.1in}

 \noindent $\bullet$ A function $\rho:\R^n \to \R$
is periodic in angle with period $\leq \alpha$ if
$$\rho(r,\theta_1,..., \theta_i + p_i,..., \theta_{n-1}) =
\rho(r,\theta_1,..., \theta_i,..., \theta_{n-1})$$ for $0\leq p_i
\leq \alpha$ and $1\leq i \leq n-1$.

\vspace{0.1in}

\noindent $\bullet$ Denote by $r^{in}(t)$, the maximal radius of a
sphere centered at the origin which is inscribed in $\Omega_t(u)$,
i.e.,
$$r^{in}(t)= \sup \{r: B_r(0) \subset \Omega_t(u)\}.$$

 \noindent $\bullet$ Denote by $r^{out}(t)$, the minimal radius
of a sphere centered at the origin, in which $\Omega_t(u)$ is
inscribed, i.e.,
$$r^{out}(t)= \inf \{r: \Omega_t(u) \subset B_r(0)  \}.$$

 \noindent $\bullet$ Denote by $\Omega^{in}$, the maximal radial region
inscribed in $\Omega(u)$, i.e., its time cross section $\Omega^{in}_
t$ is given by
$$\Omega^{in}_ t = B_{r^{in}(t)}(0) \hbox{ for all } t. $$

 \noindent $\bullet$ Denote by $0 < t_1<t_2 <...<T$, the dyadic decomposition of  the time interval
$(0,T)$  such that $t_i =(1-2^{-i})T$, i.e.,
$$ t_1 =T/2  \,\, \hbox{,  } \,\, t_{i+1}-t_i = \frac{T-t_i}{2} .$$

\noindent $\bullet$ For positive numbers $a$ and $b$,   write $ a
\approx b$ if there exist positive constants $C_1$ and $C_2$
depending only on $n$ and $M$ such that
$$C_1a \leq b \leq C_2a.$$

Below we state some properties  of caloric   functions defined in
Lipschitz domains,  a comparison principle, results on  existence of
self-similar solutions and asymptotic behavior of radial solutions,
and a regularity result for solutions with flat boundaries.

\begin{lemma} \label{ca} [ACS, Lemma 5] Let $\Omega$ be a Lipschitz domain in $Q_1(0)$
such that $0 \in \partial \Omega$. Let $u$ be a positive caloric
function in $Q_1(0)\cap \Omega$ such that $u=0$ on $\partial
\Omega$, $u(e_n,0) =m_1>0$ and $\sup_{Q_1(0)}u=m_2$. Then there
exist $a>0$ and $ \delta>0$ depending only on $n$, $L$, $m_1/m_2$
such that
$$
w_+ := u+u^{1+a}\hbox{ and } w_- := u-u^{1+a}
$$
are, respectively,  subharmonic and superharmonic in $
Q_{\delta}\cap\Omega\cap\{t=0\}.$
\end{lemma}

\begin{lemma} \label{cafmonoton} [ACS, Theorem 2] Let $\Omega$ and
$u$ be given as in Lemma~\ref{ca}, then for every $\mu \in \{\mu \in
\R^{n+1}: |\mu|=1, e_n \cdot \mu < \cos \theta\}$ where $\theta
=\cot^{-1}(L)/2$, $D_\mu u>0$ in $Q_\delta \cap \Omega$ for a
positive constant $\delta$ depending on $n$, $L$, $m_1/m_2$ and
$\|\nabla u\|_{L^2}$.
\end{lemma}

\begin{lemma} \label{caf421} [ACS, Corollary 4] Let $\Omega$ and
$u$ be given as in Lemma~\ref{ca}, then there exist positive
constants $c_1$ and $c_2$ depending on $n$ and $L$ such that
$$c_1\frac{u(x,t)}{d_{x,t}} \leq |(\nabla_x, \partial_t) u| \leq
c_2\frac{u(x,t)}{d_{x,t}}$$ for every $(x,t) \in K_r(0) \cap
\Omega$, where $d_{x,t}$ is the elliptic distance from $(x,t)$ to
$\partial \Omega$.
\end{lemma}

\begin{lemma} \label{Di262} [D, Theorem 12.2]
Let $u$ be a caloric function in $Q_\delta =Q_\delta(0)$, then there
exists a dimensional constant $C>0$ such that
$$\| \nabla u \|_{\infty, Q_{\sigma \delta}^-} \leq \frac{C}{(1-\sigma)^{n+3}\delta|Q_\delta^-|}
\int_{Q_\delta^-}|u|dxdt,$$
$$\| u_t \|_{\infty, Q_{\sigma \delta}^-} \leq \frac{C}{(1-\sigma)^{n+4}\delta^2|Q_\delta^-|}
\int_{Q_\delta^-}|u|dxdt$$ for $\sigma \in (0,1)$ where
$|Q_\delta^-|$ is the volume of $Q_\delta^-$.
\end{lemma}

\begin{lemma} \label{cak} [CV, Theorem 4.1] Let $u$ be a limit
solution of ($P$), for which the initial function $u_0 $ is
nonnegative, bounded and $|\nabla u_0| \leq M$. Then for a
dimensional constant $C_0>0$
$$|\nabla u| \leq C_0 \max\{1, M\}.
$$
\end{lemma}

\begin{lemma} \label{cp} [K, Theorem 1.3 and Theorem 2.2] Let $u$ an $v$
be, respectively, a sub- and supersolutions of (P) with  strictly
separated initial data $u_0\prec v_0$. Then the solution remain
ordered for all time, i.e.,
$$u(x,t) \prec v(x,t) \hbox{ for every }t>0.$$
\end{lemma}

\begin{lemma} \label{pro11} [CV, Proposition 1.1] Let $T>0$. Then
there exists a self similar solution $U(x,t)$ of $(P)$ in the
 form
$$U(x,t)=(T-t)^{1/2}f(|x|/(T-t)^{1/2})$$ where the profile $f(r)$
satisfies the stationary problem
$$f'' + (\frac{n-1}{r} -\frac{1}{2}r)f' +\frac{1}{2}f=0 \hbox{ for }0<r<R, $$
$$f'(0) =0 \hbox{ and } f(r)>0 \hbox{ for } 0 \leq r <R$$
with boundary conditions
$$f(R)=0 \hbox{ and } f'(R)=-1.$$
\end{lemma}

\begin{lemma} \label{ghv} [GHV, Theorem 6.6]
Let $u$ be a radial solution of ($P$) with initial data
$u_0=u_0(|x|)>0$ supported in a ball. Then
$$(T-t)^{-1/2} u(|x|,t) \rightarrow f(|x|/(T-t)^{1/2}) \hbox{ uniformly} $$
as $t \rightarrow T$ with $f$ given as in Lemma~\ref{pro11}.

\end{lemma}

\begin{lemma} \label{AW} [AW, Theorem 8.4]
Let $(u, \chi)$ be a domain variation solution of $(P)$ in
$Q_\rho:=Q_\rho(0,0)$ such that $(0,0) \in \partial \{u>0\}$. There
exists a constant $\sigma_1>0$ such that if $u(x,t)=\chi(x,t)=0$
when $(x,t) \in Q_\rho^-$ and $x_n \geq \sigma \rho$, and if
$|\nabla u| \leq 1 +\tau$ in $Q_\rho^-$  for some $\sigma\leq
\sigma_1$ and $\tau \leq \sigma_1 \sigma^2$, then the free boundary
$\partial \{u>0\}$ is in $Q_{\rho/4}^-$ the graph of a
$C^{1+\gamma,\gamma}$-function; in particular the space normal is
H\"{o}lder continuous in $Q_{\rho/4}^-$.
\end{lemma}

\section{Estimate on inner and outer radii of $\Omega$}
\label{sec3}

If $v$ is a self-similar solution (see Lemma~\ref{pro11}) with an
extinction time $T$, then the maximum of $v$ at time $t$ and the
radius of its support $\Omega_t(v)$ are constant multiples of
$\sqrt{T-t}$, i.e., there are dimensional constants $a_1$ and $a_2$ such that 
\begin{equation} \label{a1} 
\max v(\cdot,t) = a_1 \sqrt{T-t}\,\, \hbox{ and }  \,\, \Omega_t(v)
=B_{a_2\sqrt{T-t}}(0).
\end{equation}

 In this section, we prove analogous estimates, Lemma~\ref{lemma1}, on $\max
u(\cdot, t)$ and on inner, outer radii of concentric spheres which
trap the free boundary $\Gamma_t(u)$ in between.  Recall that
$r^{in}(t)$ and $r^{out}(t)$ are inner and outer radii of
$\Omega_t(u)$, i.e.,
$$r^{in}(t)= \sup \{r: B_r(0) \subset \Omega_t(u)\}$$
and
$$r^{out}(t)= \inf \{R: \Omega_t(u) \subset B_R(0)  \}.$$

 \begin{figure}
 \centering
 \includegraphics[width=0.7\textwidth]{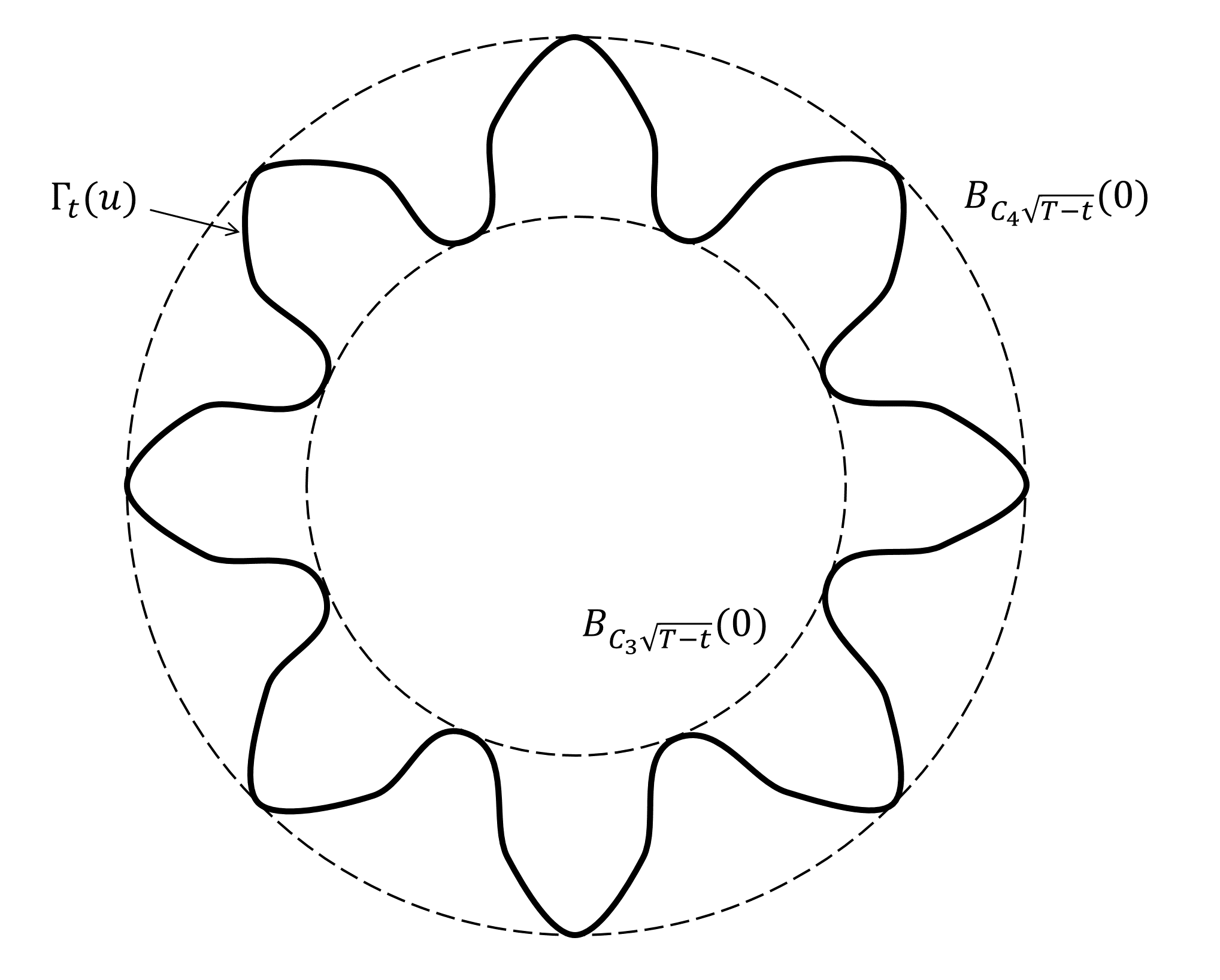} \caption{ (\ref{rt}) of  Lemma~\ref{lemma1}}
 \end{figure}

\begin{lemma} \label{lemma1} There is a constant $\alpha(n,M)>0$ depending only on $n$ and $M$ such that if
 $u_0$ is given as in Theorem~\ref{thm1} with $ \alpha <
\alpha(n, M)$, then 
\begin{equation} \label{maxu}
C_1\sqrt{T-t} \leq \max u(\cdot,t) \leq C_2\sqrt{T-t}
\end{equation}
 and
\begin{equation} \label{rt}
C_3\sqrt{T-t} \leq r^{in}(t) \leq r^{out}(t) \leq C_4\sqrt{T-t}.
\end{equation}
for constants $C_j$ ($1\leq j \leq 4$) depending only on $n$ and $M$. In fact, we can take $\alpha(n,M) =C(n)/M^9 $, $C_1= C(n)/M^5$, $C_2=C(n)M$, $C_3=C(n)/M$ and $C_4=C(n) M^3$, where $C(n)$ are dimensional constants.
\end{lemma}
\begin{proof}
We  prove the last inequality of (\ref{maxu}) and the first
inequality of (\ref{rt}) by   comparison with self similar
solutions.  Then using these inequalities, we prove the first and
the last inequality of (\ref{maxu}) and (\ref{rt}). 

\vspace{0.1in} 1. {\it Proof of } $\max u(\cdot,t) \leq
C_2\sqrt{T-t}$ ; \,\,\,  $C_3\sqrt{T-t} \leq r^{in}(t)$: Let $T$ be
the extinction time of $u$ and let $v$ be a
self-similar solution vanishing at time $T$.  Then 
$$\max v(\cdot,t) = a_1 \sqrt{T-t}\,\, \hbox{ and }  \,\, \Omega_t(v)
=B_{a_2\sqrt{T-t}}(0).$$ To find an upper bound on $ u(\cdot,t)$,
suppose that for some $x_0 \in \Omega_t(u)$
$$u(x_0,t) \geq (2a_1 + C_0Ma_2 )\sqrt{T-t}$$
where $C_0$ is a dimensional constant given as in Lemma~\ref{cak}. Then by
Lemma~\ref{cak}
 $$u(\cdot,t) \geq 2a_1 \sqrt{T-t} = 2 \max v(\cdot,t) $$   on  $B_{a_2\sqrt{T-t}
 }(x_0) = \Omega_t(v(x-x_0, t)).$
By  comparing $u$ with $v(x-x_0, t)$,  $$\max u(x,T)
> \max v(x-x_0,T) =0$$ which would contradict that $u$ vanishes at time
$T$. Hence we obtain
\begin{equation} \label{seconi}
\max u(\cdot,t) \leq C_2\sqrt{T-t}
\end{equation}
with $C_2=2a_1 + C_0Ma_2$. The first inequality of (\ref{rt}), that
is $C_3\sqrt{T-t} \leq r^{in}(t)$, can be proved similarly  for
$C_3=a_1/C_0(M+1)$  by comparing
 $u$ with a self-similar solution.

\vspace{0.1in}

2. {\it Proof of } $C_1\sqrt{T-t} \leq \max u(\cdot,t)$ {\it and }
$r^{out}(t) \leq C_4 \sqrt{T-t}$:  These inequalities will be proved
simultaneously by induction.  Let $C_4$ be a
constant depending on $M$ and $n$, which will be determined later.
Recall that $0<t_1<t_2<...<T$ is a dyadic decomposition of $(0,T)$
with $t_1=T/2$ and $t_{i+1}-t_i = (T-t_i)/2$.

\vspace{0.1 in}

{\it Claim 1. Suppose
\begin{equation} \label{C4}
r^{out}(t_i) \leq C_4 \sqrt{T-t_i}
\end{equation}
 for some $i \in \N$, then
$$\max u(\cdot, t_i) \geq \frac{C_n C_3^2}{C_4} \sqrt{T-t_i}$$
where $C_n$ is a positive dimensional constant and $C_3$ is the
constant as in the first inequality of (\ref{rt}).} For the proof of
Claim 1, we construct a Lipschitz region $\Sigma$ in $\R^n \times
[t_i, t_{i+1}]$. Since $r^{out}(t)$ is decreasing in time $t$, there
exists a decreasing function $\sigma(t)$ on $[t_i, t_{i+1}]$ such
that
\begin{itemize}
\item[(a-1)]
$\sigma(t) \geq r^{out}(t)$
\item[(a-2)]
 $\sigma(\tau)= r^{out}(\tau)$ for some $\tau \in [(t_i+t_{i+1})/2,
t_{i+1}]$
\item[(a-3)]
 $|\sigma'(t)| \leq \displaystyle{\frac{2(r^{out}(t_i)- r^{out}(t_{i+1}))}{t_{i+1}-t_i}}$.
\end{itemize}
(We can construct $\sigma(t)$ so that it is  linear on
$[(t_i+t_{i+1})/2, t_{i+1}]$ with slope $\displaystyle {
-2(r^{out}(t_i)- r^{out}(t_{i+1}))/(t_{i+1}-t_i)}$ and it is a
constant on $[t_i, (t_i+t_{i+1})/2]$.) Let $\Sigma$ be a space-time
region in $\R^n \times [t_i, t_{i+1}]$ such that its time cross
section is a ball of radius $\sigma(t)$ centered at $0$, i.e.,
$$
\Sigma_t  = B_{\sigma(t)}(0) $$
 for  $t_i \leq t \leq
t_{i+1}$. Then the properties (a-1), (a-2) and (a-3)  imply
\begin{itemize}
\item[(b-1)]  $\Omega(u) \cap  \{t_i \leq t \leq t_{i+1}\} \subset
\Sigma$
\item[(b-2)]  There exists a free boundary point
$$p \in \partial B_{\sigma(\tau)}(0) \cap \Gamma_\tau(u),$$
i.e., $(p,\tau) \in \partial \Sigma \cap \Gamma(u)$.
\item[(b-3)] $\Sigma$  is Lipschitz in space and time with a Lipschitz constant
$$L := \frac{2 \sigma(t_i)}{t_{i+1}-t_i}$$
\end{itemize}
where  (b-2) follows from (a-2) since $\partial B_{r^{out}(t)}(0)$
intersects $\Gamma_t(u)$ for all $t$.

Define a function $w_{t_i}(x)$ on $B_{\sigma(t_i)}(0)$ by
$$w_{t_i}(x) =  \left\{\begin{array} {lll} \max u (\cdot, t_i) &\hbox{ for }& x \in
B_{\sigma_0}(0)\\ \\
 \max u(\cdot, t_i) -C_0M|x| &\hbox{ for }& x \in
B_{\sigma(t_i)}(0) -B_{\sigma_0}(0)\end{array}\right.$$ where $C_0$
is the constant as in Lemma ~\ref{cak} and $\sigma_0$ is chosen so
that $w_{t_i}=0$ on $\partial B_{\sigma(t_i)}(0)$. Then by
Lemma~\ref{cak} and by $\Omega_{t_i}(u) \subset B_{r^{out}(t_i)}(0)
\subset B_{\sigma(t_i)}(0)$,
$$u(\cdot, t_i) \leq w_{t_i}(\cdot).$$
Let  $w(x,t)$ be a caloric function in $\Sigma$ such that
$$
\left\{\begin{array} {lll} \Delta w = w_t &\hbox{ in } & \Sigma \\
\\
w= w_{t_i}  &\hbox{ on } & \{t=t_i\} \\ \\
  w= 0 &\hbox{ on } & \partial \Sigma \cap \{t_i< t< t_{i+1} \}.
\end{array}\right.$$
Then by comparison, $w \geq u$ in $\Sigma$.   Since $w(p, \tau)=u(p,
\tau)=0$, the inequality $w \geq u$  implies
\begin{equation} \label{1}
|\nabla w (p, \tau)| \geq 1.
\end{equation}

Denote
$$\sigma(t_i) = \beta \sqrt{t_{i+1}-t_i} $$
for some $\beta>0$. Then
\begin{equation} \label{defL}
L= \frac{2 \beta} {\sqrt{t_{i+1} -t_i}}.
\end{equation}
Also observe that $r^{out}(t_i) \leq \sigma(t_i) \leq 2
r^{out}(t_i)$ by the construction of $\sigma(t)$. This implies
\begin{equation} \label{c3b}
 C_3 <  \beta < 4 C_4
 \end{equation}
where the first inequality follows from the first inequality of
(\ref{rt}) and the last inequality follows from the assumption
(\ref{C4}). Here we assume that $C_3 \leq 1$ without loss of
generality.

Since $\partial \Sigma =\partial\{w(x,t)>0\}$ has a Lipschitz
constant $L$, the caloric function
$$\tilde{w}(x,t) :=w(\frac{x}{L}, \tau + \frac{t}{L^2})$$ has a Lipschitz boundary
with Lipschitz constant $1$ in the region $$B_{C_3 \beta}(Lp) \times
[ -\beta^2,0].$$ Since $\beta>C_3$ and $C_3 \leq 1$, $\tilde{w}$ has
a Lipschitz boundary in a smaller region $$Q_{C_3^2}^-(Lp,0):=
B_{C_3^2}(Lp) \times [-C_3^4, 0]$$ with a Lipschitz constant $1$.
Hence by Lemma~\ref{ca}, $\tilde{w}(\cdot,0)$  is almost harmonic
near the vanishing Lipschitz boundary $\partial
B_{Lr^{out}(\tau)}(0)$.
 More precisely, there exists a constant $0<C_n<1$ depending on $n$
such that the following holds: if $h$ is a harmonic function in the
annulus $B_{Lr^{out}(\tau)}(0) - B_{Lr^{out}(\tau) -C_nC_3^2}(0)$
with
$$
h=\left\{\begin{array} {lll} 0 &\hbox{ on } & \partial
B_{Lr^{out}(\tau)}(0)\\ \\
 2\tilde{w}(\cdot, 0) &\hbox{ on } &\partial
B_{Lr^{out}(\tau) -C_nC_3^2}(0)
\end{array}\right.$$
then on $B_{Lr^{out}(\tau)}(0) - B_{Lr^{out}(\tau) -C_nC_3^2}(0)$
\begin{equation} \label{2}
 \tilde{w}(\cdot, 0) \leq h(\cdot).
\end{equation}
Combining (\ref{1}) and (\ref{2}), we obtain
\begin{equation} \label{nablah}
|\nabla h (Lp)| \geq |\nabla \tilde{w}(Lp,0)| =|\nabla w(p,\tau)|/L
\geq 1/L.
\end{equation}
 This implies
\begin{eqnarray} \label{3}
\max u(\cdot, t_i) &=& \max w(\cdot, t_i) \nonumber\\
&\geq& \max w(\cdot, \tau) \nonumber\\
&=& \max \tilde{w}(\cdot, 0) \nonumber\\
  &\geq& \max \{\tilde{w}(x, 0): x \in \partial B_{Lr^{out}(\tau) -C_nC_3^2}(0)\}
  \nonumber \\
  &=& \max h/2 \nonumber \\
  &\geq & C_nC_3^2/L \nonumber \\
&\geq & \frac{C_nC_3^2}{C_4} \sqrt{T-t_i}
\end{eqnarray}
where
 the third inequality follows from (\ref{nablah}) for a dimensional
 constant $C_n$
 and the last inequality follows from (\ref{defL}),  (\ref{c3b}) and $T-t_i =2(t_{i+1}-t_i)$.

\vspace{0.2 in}

{\it Claim 2. Suppose $r^{in}(t_i) \leq (C_4/8) \sqrt{T-t_i}$ for
$1 \leq i \leq k$, then
\begin{equation} \label{Cl2}
r^{out}(t_i) \leq C_4 \sqrt{T-t_i}
\end{equation}
for $1 \leq i \leq k$.} We prove Claim 2 by induction. Let $x_0$ be a
point in $\Omega_0(u)$ such that $\max u_0=u_0(x_0)$. Since $\max
u_0 \geq 1/M$ and $|\nabla u_0| \leq M$,
$$ u_0 \geq \frac{1}{2M} \,\, \hbox{ on }\,\,  B_{1/2M^2}(x_0).$$
 Then by comparing $u$ with a self similar solution,
\begin{equation}\label{tc}
T \geq \frac{C_n}{M^4}.
\end{equation}
 Hence
$$r^{out}(t_1) \leq r^{out}(0) \leq M \leq C M^3 \sqrt{T-t_1} \leq C_4 \sqrt{T-t_1}$$
where the first inequality follows since $\Gamma_t(u)$ shrinks in
time, the second inequality follows from the assumption on $u_0$, 
the third inequality follows from (\ref{tc}) for a dimensional constant $C>0$ and the last inequality follows if we define 
\begin{equation} \label{c4}
C_4 = C(n)M^3 
\end{equation} for a large dimensional constant $C(n)$. ($C(n)$ will be chosen in Claim 3 below.)

Now suppose that (\ref{Cl2}) holds for $i \in \{1,...,j\}$ where $j
\leq k-1$. Construct a Lipschitz region $\Sigma$ in $\R^n \times
[t_j,t_{j+1}]$ as in the proof of Claim 1 so that $$\Omega(u) \cap
\{t_j \leq t \leq t_{j+1}\} \subset \Sigma \hbox{,  }\,\,\, (p,
\tau) \in \Gamma(u) \cap
\partial \Sigma.$$ If $r^{out}(t_{j+1}) > C_4 \sqrt{T-t_{j+1}}$, then
\begin{equation} \label{rtau}
r^{out}(\tau) \geq r^{out}(t_{j+1}) > \frac{C_4}{2} \sqrt{T-\tau} >
\frac{C_4}{4} \sqrt{T-t_j} \geq 2r^{in}(t_j)
\end{equation}
where the last inequality follows from the assumption on
$r^{in}(t_j)$. Let
$$
\tilde{\Sigma}:= \Sigma - \Omega^{in}
$$
 where $\Omega^{in}$ is the region
constructed  in Section~\ref{sec2}, i.e.,  $\Omega^{in}$ is the
maximal radial region inscribed in $\Omega(u)$. Then by (\ref{rtau})
and (b-3), $\tilde{\Sigma}$ is Lipschitz in the large cube $
Q_{C_4\sqrt{T-\tau}/4}^-(p,\tau)$. Let $v(x,t)$  solve
$$
\left\{\begin{array} {lll} \Delta v = v_t &\hbox{ in } &
\tilde{\Sigma} \\ \\
v= \max_{\tilde{\Sigma}_{t_j}} u(\cdot, t_j)  &\hbox{ on } &
\{t=t_j\} \\ \\
  v= 0 &\hbox{ on } & \partial \Sigma \cap \{t_j< t< t_{j+1} \} \\
  \\
  v(\cdot, t)= \max_{\partial B_{r^{in}(t)}(0)} u(\cdot, t)
  &\hbox{ on } & \partial \Omega^{in} \cap \{t_j< t< t_{j+1}\}.
\end{array}\right.$$
Then by comparison, $v \geq u$ in $\tilde{\Sigma}$. By a similar
argument as in the proof of (\ref{3}) with  $w$ replaced by $v$, and
with (\ref{C4}) replaced by the assumption $r^{out}(t_j) \leq C_4
\sqrt{T-t_j}$, we obtain
\begin{equation} \label{vc}
\max v(\cdot, \tau) \geq \frac{C_nC_3^2}{C_4} \sqrt{T-t_j}.
\end{equation}

On the other hand, observe that for a dimensional constant $C_n$
\begin{equation} \label{111}
\max_{\partial B_{r^{in}(t)}(0)} u(\cdot, t) \leq C_n M \alpha
r^{in}(t)
\end{equation}
since $\min _{\partial B_{r^{in}(t)}(0)} u(\cdot, t) =0$,   $u$ is
periodic in angle with period $< \alpha$ and $|\nabla u| \leq C_0M$
(Lemma~\ref{cak}). Similarly,
\begin{equation} \label{tsigma}
\max_{\tilde{\Sigma}_{t_j}} u(\cdot, t_j) = \max_{\substack{
   \partial B_s(0) \\
   r^{in}(t_j) \leq s \leq r^{out}(t_j)}}
 u(\cdot, t_j)  \leq C_n
M \alpha r^{out}(t_j)
\end{equation}
since the simple connectivity of $\Omega_0(u)$ and $u_t \leq 0$
imply that $\min _{\partial B_{s}(0)} u(\cdot, t_j) =0$ for
$r^{in}(t_j) \leq s \leq r^{out}(t_j)$. Hence
$$\max v \leq C_n M \alpha r^{out}(t_j) \leq C_n M \alpha C_4 \sqrt{T-t_j}$$
where the last inequality follows from the assumption on
$r^{out}(t_j)$. If 
$$\alpha< \alpha(n,M) := C_n C_3^2/C_4^2M = C(n)/M^9,$$   then
the above upper bound on $\max v$ would contradict (\ref{vc}). Hence
we conclude
$$r^{out}(t_{j+1}) \leq C_4 \sqrt{T-t_{j+1}}.$$

\vspace{0.2in}

{\it Claim 3. Recall that $C_4 = C(n)M^3$. If $C(n)$ is   large, 
\begin{equation} \label{cl3}
r^{in}(t_i) \leq \frac{C_4}{8} \sqrt{T-t_i}
\end{equation}
for all $i\geq 1$. }
For $i=1$, (\ref{cl3}) follows from (\ref{tc}) and $r^{in}(t_1)\leq
r^{in}(0)=1$. Now suppose that (\ref{cl3}) holds for $1\leq i \leq
j$ and not for $i=j+1$, i.e.,
\begin{equation} \label{e}
 r^{in}(t_{j+1}):=r_0> \frac{C_4}{8}
\sqrt{T-t_{j+1}}.
\end{equation}
Since (\ref{cl3}) holds for $1\leq i \leq j$, Claim 2 implies
$$r^{out}(t_i) \leq C_4 \sqrt{T-t_i} \hbox{ for } 1\leq i \leq j.$$
Then by Claim 1,
$$\max u(\cdot, t_i) \geq \frac{C_nC_3^2}{C_4}
\sqrt{T-t_i} \hbox{ for } 1\leq i \leq j .$$ Since (\ref{tsigma})
implies that $\max u(\cdot, t_i)$ ($1\leq i \leq j$) is taken inside
the maximal radial region $\Omega^{in} \subset \Omega(u)$, and since
$T\geq C(n,M)$,
\begin{equation} \label{c'c}
u(0, t_i) \geq \frac{CC_3^2}{C_4} \sqrt{T-t_i} \hbox{ for } 1\leq i
\leq j
\end{equation}
where $C$ is a constant depending on $n$ and $M$.

 Let  $k =\min \{k \in \{1,..., j\}: t_{j+1}-t_k \leq r_0^2\}$ where $r_0=r^{in}(t_{j+1})$. (Here observe
that $t_{j+1} -t_j =T-t_{j+1} < r_0^2$ by (\ref{e}).) Then
\begin{equation}\label{ee}
B_{r_0}(0) \times [t_k, t_{j+1}] \subset Q^-_{r_0}(0, t_{j+1})
\subset \Omega(u).
\end{equation}
  Observe that
\begin{equation}\label{eee}
 \left\{\begin{array} {lll}
 t_{j+1}-t_{k} \geq (t_{j+1} -t_{k-1})/3 \geq r_0^2/3  &\hbox{ if } &k \neq 1 \\ \\
t_{j+1}-t_{k} \geq T/4 \geq  C(n,M) \geq C(n,M) r_0^2 &\hbox{ if } &
k=1.
\end{array}\right.
\end{equation}
 Then by  (\ref{e}), (\ref{c'c}), (\ref{ee}) and (\ref{eee})
\begin{eqnarray*}
\min_{B_{r_0/2}(0)} u(\cdot, t_{j+1}) &\geq& C u(0, t_k) \\
&\geq& \frac{C C_3^2}{C_4} \sqrt{T-t_k} \\ &\geq& \frac{C
C_3^2}{C_4} \sqrt{t_{j+1}-t_k}
\\&\geq& \frac{C C_3^2 }{C_4}\cdot r_0\\
&\geq& C C_3^2 \sqrt{T-t_{j+1}}
\end{eqnarray*}
where $C$ denote constants depending on $n$ and $M$. In other words,
$u(\cdot, t_{j+1})$ has a lower bound $C C_3^2 \sqrt{T-t_{j+1}}$  on
the large ball $B_{r_0/2}(0)$ with the radius
$$r_0/2 \geq C_4 \sqrt{T-t_{j+1}}/8.$$
Then if $C_4 = C(n)M^3$ for a large dimensional constant $C(n)>0$, then 
 $u$
would have an extinction time larger than $T$, contradicting
$u(\cdot, T) \equiv 0$. Hence we conclude $r^{out}(t_{j+1}) \leq
(C_4/8)\sqrt{T-t_{j+1}}$.

\end{proof}

If $U$ is a self-similar solution with an extinction time $T$, then
the normal velocity of its free boundary at time $t$  is comparable
to $1/ \sqrt{T-t} \approx 1/r^{in}(t)$, and hence $\Gamma(U)$ is
Lipschitz in a parabolic scaling in each $Q_{r^{in}(t)}(x, t)$,
$(x,t) \in \Gamma(U)$. Recall that $\Omega^{in}=\Omega^{in}(u)$ is
the maximal radial subregion of $\Omega(u)$, i.e., its time cross
section $\Omega^{in}_t$ is given by
$$\Omega^{in}_t = B_{r^{in}(t)}(0) \hbox{ for } 0 \leq t \leq T.$$
 In the next lemma, we
prove an analogous result that the average normal velocity of
$\partial \Omega^{in}$  is bounded above by $C(n,M)/ \sqrt{T-t}$ on
each time interval $[t, t+\alpha r^{in}(t)^2]$ for $t \geq T/2$.
This gives that the inner region $\Omega^{in}$ can be approximated
by a subregion $\Omega_1$ which is Lipschitz in a parabolic scaling.

\begin{lemma} \label{lemma2}
Let $u_0$ be given as in Theorem~\ref{thm1} with $ \alpha <
\alpha(n, M)$, where $\alpha(n,M)>0$ is the constant as in Lemma~\ref{lemma1}. Then there
exists a space-time region $\Omega_1 \subset \Omega^{in}$, which is
radial in space and satisfies the following conditions:
\begin{itemize}
\item[(i)] For $i \geq 1$ (i.e., for $t_i \geq T/2$)
$$S_i :=\Omega_1 \cap \{t_i \leq t\leq t_{i+1}\} $$
is Lipschitz in a parabolic scaling with a Lipschitz constant
$C(n,M)$, i.e., the normal velocity of $\partial S_i$ is bounded
above by $C(n,M)/r^{in}(t_i)$.
\item[(ii)] $\partial S_i$ is located in the $C(n,M)\alpha
r^{in}(t_i)$-neighborhood of $\partial \Omega^{in}$.
\end{itemize}
In particular,
\begin{equation} \label{22222}
u \leq C(n,M)\alpha r^{in}(t_i) \hbox{ on }\partial S_i.
\end{equation}
\end{lemma}
\begin{proof}
Suppose that we have a subregion $\Omega_1$ of $\Omega^{in}$
satisfying the conditions (i) and (ii).  Then for a constant $C$
depending on $n$ and $M$,
$$\max_{\partial S_i} u \leq \max_{\partial \Omega^{in} \cap \{t_i \leq t \leq
t_{i+1}\}} u+ C \alpha r^{in}(t_i) \leq C \alpha r^{in}(t_i)$$ where
the first inequality follows from $|\nabla u| \leq C_0M$
(Lemma~\ref{cak}), and the last inequality follows from (\ref{111}).

Denote by
 $V_{I}$, the average normal velocity of $\partial
\Omega^{in}$ on the time interval $I$. Decompose the time interval
$[t_{i}, t_{i+1}] $ into  subintervals of length $\alpha
r^{in}(t_i)^2$, i.e., let
$$
t_i =t_{i0} < t_{i1}= t_{i0} +\alpha r^{in}(t_i)^2 < t_{i2}= t_{i1}
+\alpha r^{in}(t_i)^2 < ...< t_{ik} =t_{i+1}.
$$
 For the construction of  $\Omega_1$ satisfying the condition of the lemma, it suffices to prove that
\begin{equation} \label{5}
 V_{[t_{i
j},\,\, t_{i\,j+1}]} \leq C(n,M)/r^{in}(t_i).
 \end{equation}
More precisely, given the estimate (\ref{5}), the Lipschitz
subregion $\Omega_1$ can be constructed so that  $\partial \Omega_1$
and $\partial \Omega^{in}$ intersect at times $t=s_m$ with $\{s_m\}
\subset [t_i, t_{i+1}]$ and
 $|s_{m+1}-s_m|\leq \alpha
r^{in}(t_i)^2$.

Below we prove (\ref{5}). Let $i \geq 1$. By Lemma~\ref{lemma1}, the
average velocity of $\partial \Omega^{in}$ on $[t_{i-1}, t_i]$, that
is $V_{[t_{i-1}, t_i]}$, is bounded above by $$C(n,M)/\sqrt{T-t_i}
\approx C(n,M)/r^{in}(t_i).$$ Then there exists $\tau \in
[(t_{i-1}+t_i)/2, t_i]$ such that $V_{[t, \tau]} \leq
C(n,M)/r^{in}(t_i)$ for all $t \in [t_{i-1}, \tau]$. In particular,
$V_{[\tau-\alpha r^{in}(t_i)^2, \tau]} \leq C(n,M)/r^{in}(t_i)$. Let
$$\Sigma = B_{r^{in}(\tau)}(0) \times [\tau-\alpha r^{in}(t_i)^2,
\tau].$$ Denote $\tilde{\tau} = \tau-\alpha r^{in}(t_i)^2$ and let
$\tilde{\phi}(x,t)$ be the maximal radial function such that
$\tilde{\phi}(x,t) \leq u(x, t)$. Let $\psi(x,t)$ be a solution of
$$
 \left\{\begin{array} {lll}
 \Delta \psi =\psi_t  &\hbox{ in  } & \Sigma \\ \\
\psi= 0   &\hbox{ on  } & \partial \Sigma \\ \\
\psi =  \tilde{\phi} &\hbox{ on  } &  \{t= \tilde{\tau}\}.
\end{array}\right.
$$
Then $\psi \leq u$ and by Lemma~\ref{ca}, $\psi(\cdot, \tau)$ is
almost harmonic in the $c \sqrt{\alpha} r^{in}(t_i)$-neighborhood of
$\partial \Omega^{in}_{\tau}$. Observe  that  $\partial \Sigma$ is
located in the $C\alpha r^{in}(t_i)$-neighborhood of $\partial
\Omega^{in}$, since $V_{[\tilde{\tau}, \tau]} \leq
C(n,M)/r^{in}(t_i)$. Then by a similar argument as in (\ref{111})
and (\ref{tsigma}),
$$
u -\psi =u \leq C(n,M) \alpha r^{in}(t_i) \,\,\hbox{  in }\,\,
\Omega \cap \{\tilde{\tau} \leq t\leq \tau \} -\Sigma.
$$
Also since the initial perturbation $\rho \leq \alpha \leq \alpha M
\max \phi_0 $,
$$
\psi \geq (1 -C(n,M)\alpha)u \,\,\hbox{ on }\,\,
\Sigma_{\tilde{\tau}}.
$$
Hence on $\partial B_{(1 - c \sqrt{\alpha}) r^{in}(t_i)}(0)$,
\begin{equation}\label{12}
\psi(\cdot, \tau) \geq u(\cdot, \tau ) -C(n,M) \alpha r^{in}(t_i)
\geq (1-C\sqrt{\alpha})u(\cdot, \tau)
\end{equation}
 and  $|\nabla  \psi| \geq
1-C\sqrt{\alpha}$ on $\partial \Omega^{in}_{ \tau}$ for a constant
$C$ depending on $n$ and $M$. (Otherwise, there would exist a free
boundary point at which $|\nabla u| <1$.) Hence $u(\cdot, \tau)$ is
bounded below by a function $\psi(\cdot, \tau)$ which is almost
harmonic in the $c \sqrt{\alpha} r^{in}(t_i)$-neighborhood of
$\partial \Omega^{in}_{\tau}$ with $|\nabla \psi| \geq
1-C\sqrt{\alpha}$ on $\partial \Omega^{in}_\tau$. In fact, this
lower bound can be obtained with $\alpha$ replaced by any number
$\geq \alpha$. Then by a barrier argument, $u(\cdot, \tau + \alpha
r^{in}(t_i)^2)
>0$ on $B_{r^{in}(\tau) - C \alpha r^{in}(t_i)} (0)$, i.e.,
$$
V_{[\tau, \tau+\alpha r^{in}(t_i)^2]} \leq C(n,M)/r^{in}(t_i).
$$
Next, we take $\tau_1 \in [\tau+\alpha r^{in}(t_i)^2 /2, \tau
+\alpha r^{in}(t_i)^2] $ such that $V_{[t ,\tau_1]} \leq
C(n,M)/r^{in}(t_i)$ for all $t \in [\tau, \tau_1]$. Then by a
similar argument, we obtain (\ref{5}) on the interval $[\tau_1,
\tau_1+\alpha r^{in}(t_i)^2]$. By induction, (\ref{5}) holds on each
time interval with length $\alpha r^{in}(t_i)^2$.
\end{proof}

\section{$\alpha$-flatness of free boundary } \label{sec4}

In this section we prove the $\alpha$-flatness of the free boundary
$\Gamma(u)$. More precisely, for $(x,t) \in \Gamma (u)$ we locate
the free boundary part $\Gamma(u) \cap Q^-_{r(t)}(x,t)$ in the $K
\alpha r(t)$-neighborhood of the Lipschitz boundary $\partial
\Omega_1$. Here $\alpha$ is
 the size of the initial perturbation and $K$ is a
 constant depending on $n$ and $M$.
 
 \vspace{0.1in}
 
 $\bullet$  Throughout the rest of the paper, we let  $u_0$ be given as in
 Theorem~\ref{thm1} with $ \alpha \leq \alpha(n, M)$, where $\alpha(n,M)>0$ is a constant depending only on  $n$ and $M$.

\vspace{0.1in}

$\bullet$ Denote by $r(t)$, the radius of the time cross-section of
$\Omega_1$ at time $t$, i.e.,
$$\Omega_{1t}=B_{r(t)}(0) \hbox{ for } 0 \leq t \leq T .$$

\vspace{0.1in}
\noindent
 Note that by Lemma~\ref{lemma2}, $
r(t) \leq  r^{in}(t) \leq (1+C\alpha) r(t) $ for a constant $C$
depending on $n$ and $M$.

 \begin{figure}
 \centering
 \includegraphics[width=0.7\textwidth]{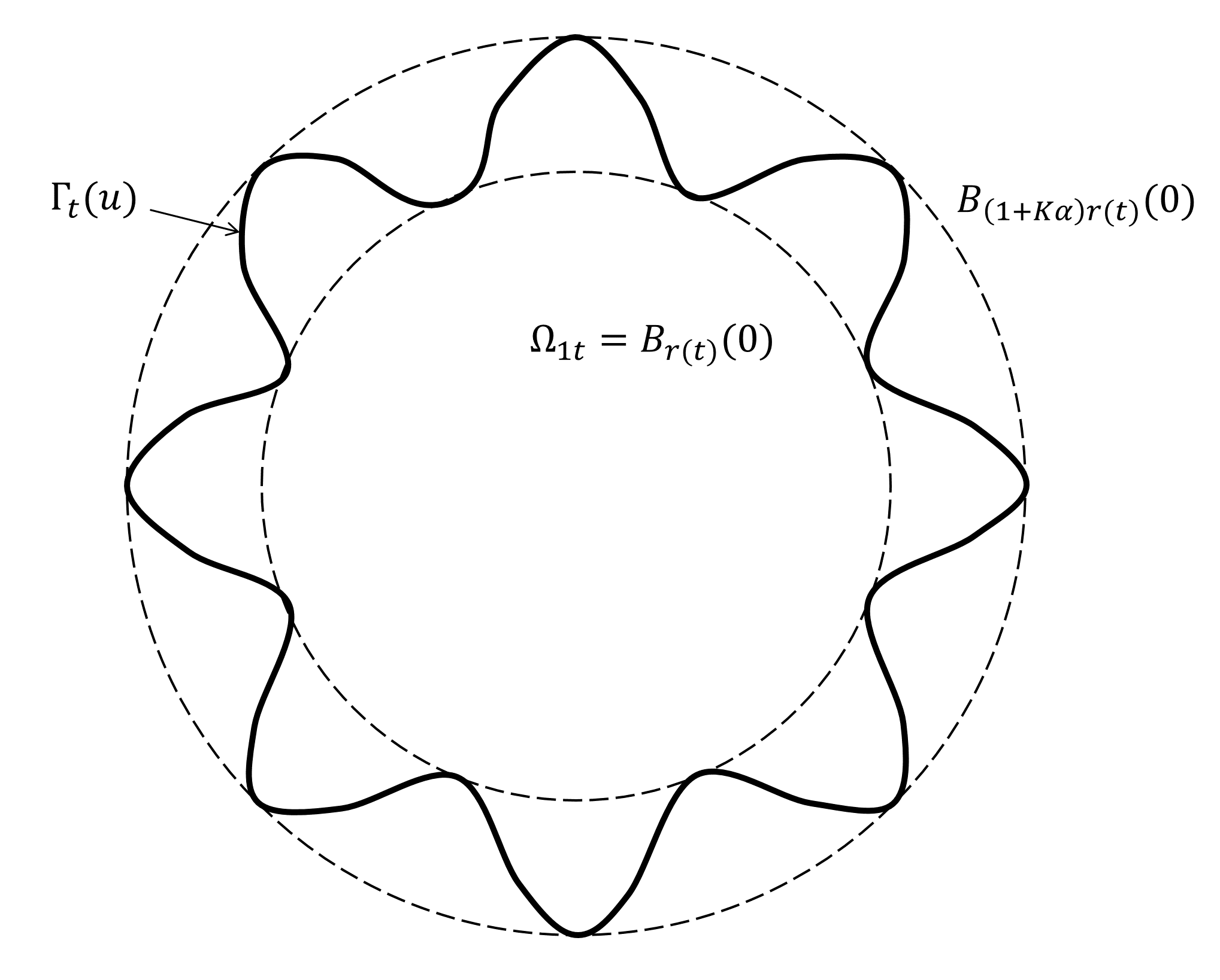} \caption{ Lemma~\ref{lemma3}}
 \end{figure}

\begin{lemma} \label{lemma3}
There is a constant $K = K(n,M)>0$ such that 
$$r(t) \leq r^{out}(t) \leq (1+K \alpha)r(t)$$
for $t_2 \leq t\leq T$.  In other words,  $\Gamma_t(u)$ is contained in the annulus
$B_{(1 +K \alpha) r(t)}(0) - B_{r(t)}(0) $ for $t_2 \leq t\leq T$. (See Figure 4.)
\end{lemma}

\begin{proof}
Let $K$ be a constant depending  on $n$ and $M$,
which will be chosen later. Let $\Omega_2$ be a space time region
containing $\Omega_1$ such that its time cross-section $\Omega_{2t}$
is given by
$$\Omega_{2t}  = (1+K \alpha) \Omega_{1t} := B_{(1+K \alpha)r(t)}(0)$$
 for $0\leq t \leq T$.
Since $\Omega_1$ is Lipschitz in a parabolic scaling for $t\geq
t_1$, $\Omega_2$ is also Lipschitz in a parabolic scaling for $t
\geq t_1$. Now  modify $\Omega_2$ as below for $0 \leq t \leq t_2$
so that it is Lipschitz for all $t \geq 0$. Since $r^{out}(t_1)
\approx r^{out}(t_2) \approx 2$, we can construct $\tilde{\Omega}_2$
satisfying
\begin{itemize}
\item[a.] $\tilde{\Omega}_2 \cap \{0 \leq t\leq t_1\} = B_2(0)
\times [0, t_1]$
\item[b.] $\tilde{\Omega}_2 \cap \{t_1 \leq t\leq t_2\} \supset \Omega_2 \cap \{t_1 \leq t\leq t_2\} $
\item[c.] $\tilde{\Omega}_2 \cap \{t_2 \leq t < T\}  =
\Omega_2 \cap \{t_2 \leq t < T\} $
\item[d.] $\tilde{\Omega}_2$ is Lipschitz in a parabolic scaling.
 \end{itemize}

Let $w$ be a caloric function in $\tilde{\Omega}_2 - \Omega_1$ such
that
$$
\left\{\begin{array} {lll} \Delta w = w_t &\hbox{ in } &
\tilde{\Omega}_2 -\Omega_1 \\ \\
w=u  &\hbox{ on } & \{t=0\} \cup  (\partial \Omega_1 \cap\{t>0\} )\\
\\
  w= 0 &\hbox{ on } & \partial\tilde{\Omega}_2 \cap \{t>0\}.
\end{array}\right.$$
For the proof of the lemma, it suffices to prove $u\leq w$ since
this inequality would imply that the free boundary $\Gamma(u)$ in
contained in $\Omega_2$ for $t \geq t_2$, i.e., the outer radius
$r^{out}(t) \leq (1+K\alpha) r(t)$ for $t \geq t_2$. Below we prove
$u\leq w$.

\vspace{0.2in}

1. Since the free boundary of $u$ shrinks in time,
\begin{equation} \label{222}
u\leq w \hbox{ for } 0 \leq t\leq t_1.
\end{equation}

\vspace{0.1in}

2. Using  (\ref{22222}) and the Lipschitz property of
$\tilde{\Omega}_2$, we will show that
\begin{equation} \label{24} u \leq w \hbox{ for }t_1 \leq t \leq
t_4.
\end{equation}
Since $u(\cdot, t_1) \leq w(\cdot, t_1)$ (see (\ref{222})) and $u=w$
on $\partial \Omega_1$, it suffices to prove that $w$ is a
supersolution of ($P$)  for $t_1< t< t_4$. Since $u_0 =\rho \leq
\alpha$ on $(\tilde{\Omega}_2 -\Omega_1) \cap \{t=0\}$, the bound
(\ref{22222}) and the construction of $w$ yield
\begin{equation} \label{23}
\max w \leq C_0\alpha
\end{equation}
 for $C_0 =C_0(n,M)$. On the other hand, since $\tilde{\Omega}_2 $ is
 Lipschitz in a parabolic scaling for $t \geq 0$,  $w(\cdot, t)$ is almost
 harmonic near its vanishing boundary $\partial \tilde{\Omega}_{2t}$ for $t \geq t_1$
  (see Lemma~\ref{ca}). Observe that for $t_2 \leq t\leq t_4$
  $$\Omega_t(w)= \tilde{\Omega}_{2 t} - \Omega_{1 t} = B_{(1+K \alpha) r(t)}(0)
  - B_{r(t)}(0) $$
  and for $t_1 \leq t\leq t_2$
 $$\Omega_t(w) \supset B_{(1+K \alpha) r(t)}(0)
  - B_{r(t)}(0) $$  where
 $ r(t) \approx 1$  for $ t_1 \leq t \leq t_4$ (see (\ref{rt}) and
(\ref{tc})). Hence we can observe that if $K=K(n,M)$ is chosen large, then the almost harmonicity of $w$ with its
upper bound (\ref{23}) implies that $w$ is bounded from above by a
radial linear function with a small slope so that
\begin{equation} \label{233}
|\nabla w| <1 \hbox{ on }
\partial \tilde{\Omega}_2 \cap \{t_1 \leq t\leq t_4\}. \end{equation}
Hence $w$ is a supersolution for $t_1< t< t_4$ and (\ref{24})
follows.

 \vspace{0.2in}

3. Now suppose
 \begin{equation} \label{26}
u \leq w \hbox{  for  } 0 \leq t \leq t_i
\end{equation}
for a fixed  $i \geq 4$ and  we  show
$$u \leq w \hbox{  for }
t_i \leq t \leq t_{i+1}.$$ First, observe that  (\ref{26}) implies
 the free
 boundary $\Gamma_{t_{i-2}}(u)$  is trapped between $$\partial
 \Omega_{1 t_{i-2}}=\partial B_{r(t_{i-2})}(0) \hbox{ and }\partial
 \tilde{\Omega}_{2 t_{i-2}}=\partial B_{(1+K \alpha)r(t_{i-2})}(0).$$
In other words, the inner and outer boundaries of $\tilde{\Omega}_2
- \Omega_{1}$ at time $t=t_{i-2}$ are located within a distance $K
\alpha r(t_{i-2})$ from the free boundary of $u$. Then since
$|\nabla u |\leq C_0 M$, we obtain
\begin{equation} \label{27} u \leq C K \alpha r(t_{i-2})   \hbox{ on } (\tilde{\Omega}_2- \Omega_1)
\cap \{t=t_{i-2}\}
\end{equation}
 for some $C =C(n,M)$. Also by (\ref{22222}),
\begin{equation} \label{28} u \leq C  \alpha r(t) \leq C  \alpha r(t_{i-2})
 \hbox{ on }\partial \Omega_1 \cap \{t_{i-2}\leq t \leq t_{i+1}\}.
\end{equation}
 Now  construct two caloric functions $w_1$ and $w_2$ in $$\Pi :=
(\tilde{\Omega}_2-\Omega_1) \cap \{t_{i-2} \leq t \leq t_{i+1}\}$$
such that
$$
\left\{\begin{array} {lll} \Delta w_1 = \partial w_1 /\partial t
&\hbox{ in } & \Pi \\ \\
 w_1=u  &\hbox{ on } &\{t=t_{i-2}\}
\\ \\
 w_1= 0 &\hbox{ on } & \partial \Omega_1 \cup \partial
 \tilde{\Omega}_2
\end{array}\right.$$ and that
$$\left\{\begin{array} {lll} \Delta w_2 = \partial w_2 /\partial
t &\hbox{ in } & \Pi\\ \\
 w_2=0  &\hbox{ on } &  \{t=t_{i-2}\} \cup \partial
 \tilde{\Omega}_2
\\ \\
 w_2= u &\hbox{ on } & \partial \Omega_1.
\end{array}\right.$$
Below we prove $$u \leq w_1+w_2  $$ in $\Pi \cap \{t_i \leq t \leq
t_{i+1}\}$ by showing that $w_1+w_2$ is a supersolution of ($P$) for
$t_i \leq t \leq t_{i+1}$. On $\Pi \cap \{t=t_{i-1}\}$,
\begin{eqnarray}
  w_1 + w_2 &\leq& \max w_1(\cdot, t_{i-1}) +  C \alpha r(t_{i-2})
\nonumber \\ &\leq&  2 C \alpha r(t_{i-2})\label{255}
 \end{eqnarray}
where  the first inequality follows from (\ref{28}) and the last
inequality follows from (\ref{27}) if $\alpha \leq \alpha(n, M)$ for
a constant $\alpha(n, M)>0$ depending on $n$ and $M$, since the time cross-section
of the domain $\Pi$, that is $B_{(1+K\alpha)r(t_{i-2})}(0) -
B_{r(t_{i-2})}(0)$,  has a small thickness $K\alpha
r(t)$ for $t_{i-2} \leq t \leq t_{i-1}$. By (\ref{28}) and
(\ref{255}),
$$
 w_1+w_2 \leq C(n,M) \alpha r(t_{i-2})
$$
on the parabolic boundary of $\Pi \cap \{t_{i-1} \leq t \leq
t_{i+1}\}$ and hence
\begin{equation}\label{2555}
\max_{\Pi \cap \{t_{i-1} \leq t \leq t_{i+1}\}}w_1+w_2 \leq C(n,M)
\alpha r(t_{i-2}).
\end{equation}
 Then by a similar argument as in (\ref{233}) with (\ref{23})
replaced by (\ref{2555}), we obtain
$$ |\nabla (w_1+w_2)| <1  $$
 on $ \partial \tilde{\Omega}_2 \cap \{t_{i} \leq  t \leq t_{i+1}\}$ if $\alpha \leq \alpha(n,M)$. 
Hence we conclude $w_1+w_2$ is  a supersolution of ($P$) for $t_{i}
\leq t \leq t_{i+1}$.

By the construction of $w_1$ and $w_2$,
$$ u =w_1+w_2
$$ on the inner lateral boundary $\partial \Omega_1 \cap \{t_i \leq t \leq t_{i+1}\} $ of $\Pi \cap \{t_i \leq t \leq t_{i+1}\}$.
Also
$$  u(\cdot, t_i)  \leq (w_1 + w_2)(\cdot, t_i) $$
since the assumption (\ref{26}) implies  $\Omega(u) \cap \{t_{i-2}
\leq t \leq t_i\}  \subset \tilde{\Omega}_2$. Hence we conclude
$$u \leq w_1 +w_2 $$
in $\Pi \cap \{ t_i \leq t \leq t_{i+1}\}$. This implies that the
free boundary  $\Gamma_t(u)$ is contained in $\tilde{\Omega}_2$ for
$t_i \leq t \leq t_{i+1}$. Now recall that the caloric function $w$
has positive set $ \tilde{\Omega}_2 -\Omega_1$ with $w =u$ on
$\partial \Omega_1$ and $w=0$ on $\partial \tilde{\Omega}_2$. Also
by (\ref{26}), $u \leq w$ at time $t=t_i$. Hence by comparison,
$$u \leq w \hbox{ for } t_i \leq t \leq t_{i+1}.$$
\end{proof}

\section{Interior approximation of $u$ by a radial function $\phi$ }
\label{sec5}

In this section,  we approximate a  solution $u$ by a radial
function $\phi$ in the inner region,   $\alpha^{2/3}
r(t)$-away from the boundary $
\partial B_{r(t)}(0)$. 

\vspace{0.1in}

 \noindent $\bullet$ Let $\phi$ be a radially
symmetric function defined in  $\Omega_1$ such that on each time
interval $[t_{i}, t_{i+1})$, $\phi(x,t)$ solves
$$
\left\{\begin{array} {lll}
 \Delta \phi=\phi_t  &\hbox{ in } & \Omega_1 \cap   \{t_i < t <
 t_{i+1}\}
\\ \\
 \phi= 0 &\hbox{ on } & \partial \Omega_1 \cap \{t_i < t <t_{i+1}\}
\\  \\
\phi(x,t_i) =  \phi(|x|, t_i)=\displaystyle{\min_{\{y:|y|=|x|\}} u(y
, t_i)} &\hbox{ on } & \Omega_1 \cap \{t=t_i \}.
\end{array}\right.$$
 In other words,
\begin{itemize}
\item[(i)] $\phi(\cdot, t_i)$ is
the maximal radial function $\leq u$.
\item[(ii)]  $\phi(x,t)$ is caloric in  $\Omega_1 \cap \{t_i < t
<t_{i+1}\}$
 with  $\phi=0$ on $\partial \Omega_1 \cap \{ t_i
< t <t_{i+1}\} $, and  $\phi=\phi(\cdot, t_i)$ on $\Omega_1 \cap
\{t=t_i\}$.
\end{itemize}
Note that $\phi$ need not be continuous at $t=t_i$. By comparison,
$\phi \leq u$.

\begin{lemma} \label{lemma4}
\label{lemma24} 

For $\e=2/3$ and $t\in [t_2, T)$, 
\begin{equation} \label{in}
\phi(\cdot, t) \leq u(\cdot,t) \leq (1+ C\alpha^{1-\e}) \phi(\cdot,
t)
\end{equation}
on $B_{(1-\alpha^\e)r(t)}(0)$, where $C>0$ is a constant depending on $n$ and $M$.
\end{lemma}
\begin{proof}
Fix $\tau \in [t_i, t_{i+1})$ for  $i \geq 2$. Let $w$ solve
$$
\left\{\begin{array} {lll}
 \Delta w=w_t  &\hbox{ in } & \Omega_1 \cap \{t_{i-1} < t< \tau\}
\\ \\ w = 0 &\hbox{ on } & \partial \Omega_1 \cap \{t_{i-1} < t< \tau\}
\\ \\ w =\phi &\hbox{ on } &  \Omega_1 \cap \{t=t_{i-1}\}.
\end{array}\right.$$
Observe that  $w$ is a radially symmetric  function with $w \leq u$.
Since $\phi(\cdot, t_i)$ is the maximal radial function $\leq
u(\cdot, t_i)$, we obtain
$$w(\cdot, t_i) \leq \phi(\cdot, t_i).$$
Then by comparison
\begin{equation} \label{refrefre}
w \leq \phi \,\hbox{ for }\, t_{i-1} \leq t \leq \tau.
\end{equation}
On the other hand, by a similar argument as in the proof of
 (\ref{111}) we obtain
\begin{equation} \label{refref}
u(\cdot, t_i) -\phi(\cdot, t_i) \leq C\alpha r(t_i)
\end{equation}
 for
$C=C(n,M)>0$ and $i \in \N$ since $u$ is periodic in angle with
period $\leq \alpha$,   $|\nabla u| \leq C_0M$ and $\phi(\cdot,
t_i)$ is the maximal radial function $\leq u(\cdot, t_i)$. Hence for
some  $C=C(n,M)>0$
\begin{eqnarray*}
Cr(t_{i-1}) &\leq& \max u(\cdot, t_{i-1}) \\ &\leq& 2 \max
\phi(\cdot, t_{i-1}) =2\max w(\cdot, t_{i-1})
\end{eqnarray*}
  where the first inequality follows from
 Lemma~\ref{lemma1} and the second inequality follows from
 (\ref{refref}). The above inequality implies that for  $C=C(n,M)>0$
\begin{equation} \label{31}
Cr(\tau) \leq \max w(\cdot, \tau)
\end{equation}
 since  $r(t_{i-1}) \approx
 r(t)$ for $t_{i-1} \leq t \leq \tau$ (Lemma~\ref{lemma1}),  $|\nabla u| \leq C_0M$
 (Lemma~\ref{cak}) and $\Omega_1$ is Lipschitz in a parabolic
 scaling.  The
Lipschitz property of $\Omega_1 \cap \{t_{i-1}\leq t\leq \tau\}$
implies that
 $w(\cdot, \tau)$  is almost harmonic near
 $\partial \Omega_{1 \tau}$ (Lemma~\ref{ca}). Then by a similar
 reasoning as in (\ref{233}) with the lower bound (\ref{31}), we obtain that $w(|x|,\tau)$ is bounded
 from below by a radial linear function  vanishing on $|x|=r(\tau)$,
  with slope $c=c(n,M)$. Hence
  on $ B_{(1-\alpha^\e)r(\tau)}(0)$
\begin{equation} \label{32}
 c(n,M)\alpha^\e r(\tau) \leq  w(\cdot, \tau)  \leq \phi(\cdot,\tau)
 \end{equation}
where the last inequality follows from (\ref{refrefre}).

 Now observe that by (\ref{refref})
 $$u(\cdot, t_{i}) -\phi(\cdot, t_{i}) \leq C\alpha
 r(t_{i})$$ and on $\partial \Omega_1 \cap \{ t_i <t< \tau\}$
$$u -\phi =u \leq C\alpha r(t_i) $$
where the inequality follows from (\ref{22222}) and
Lemma~\ref{lemma1}. Hence by comparison
\begin{equation} \label{refrefr}
u(\cdot, t) -\phi(\cdot, t)   \leq   C\alpha r(t_i)
\end{equation}
 for $t_i \leq t
\leq \tau$. Conclude that on $ B_{(1-\alpha^\e)r(\tau)}(0)$
\begin{eqnarray*}
u(\cdot, \tau) -\phi(\cdot, \tau)  & \leq &  C\alpha r(t_i) \\
& \leq & C \alpha r(\tau) \\
& \leq & C \alpha^{1-\e} \phi(\cdot,\tau)
\end{eqnarray*}   for $C=C(n,M)>0$ where the first inequality follows from (\ref{refrefr}),
the second inequality follows from Lemma~\ref{lemma1} and the last inequality
follows from (\ref{32}).
\end{proof}

\section{Interior improvement by flatness of boundary} \label{sec6}

We improve the interior estimate in Lemma~\ref{lemma4}, using the
flatness of the free boundary, that is Lemma~\ref{lemma3}. More
precisely, the constant $C$ in the interior estimate  (\ref{in})
improves in time, up to an order determined by the `flatness
constant' of the free boundary.

\vspace{0.1in}

 \noindent $\bullet$ Let $\tilde{\phi}$ be a radially  symmetric
 function defined in $\Omega_1$ such that
$$ \tilde{\phi}(x, t) = \tilde{\phi}(|x|, t)=\min_{\{y:|y|=|x|\}} u(y , t).$$
In other words, $\tilde{\phi}$ is
 the maximal radial function in
 $\Omega_1$ such that $\tilde{\phi}\leq
u$. Note that $\phi(\cdot, t_i)=\tilde{\phi}(\cdot, t_i)$ for $i \in
\N$, and $\phi \leq \tilde{\phi}$ since $\phi$ is radial with
$\phi\leq u$.

\begin{lemma} \label{lemma5}
Let $\e =2/3$. Assume
\begin{itemize}
\item[(a)] (\ref{in}) holds at time $t=t_i$, i.e.,  on $B_{(1-\alpha^\e)r(t_i)}(0)$
 $$ u(\cdot,t_i) \leq (1+ C\alpha^{1-\e}) \phi(\cdot,
t_i)$$.
\item[(b)]  For  $t_{i} \leq t \leq t_{i+1}$, $\Gamma_t(u)$ is contained in $B_{(1+K\alpha)r(t)}(0)
-B_{r(t)}(0)$  for a constant $K $ satisfying
\begin{equation} \label{K}
1 \leq K<\alpha^{\frac{\e-1}{2}}C
\end{equation}
where $C$ is given as in (a).
\item[(c)] On $B_{r(t_{i})}(0) - B_{(1-\alpha^\e)r(t_{i})}(0)$,
$$u(\cdot, t_{i}) \leq \phi(\cdot,
t_{i}) +  L(C+K) \alpha r(t_{i})$$
where $L$ is a positive constant depending on $n$ and $M$;  $C$ and
$K$ are given as in (a) and (b).
\end{itemize}
 Then for a constant $0<h=h(n,M)<1$,
the condition (a) holds with $C$ replaced by $hC$ at time
$t=t_{i+1}$, i.e.,
\begin{equation} \label{lem5}
u (\cdot, t_{i+1}) \leq (1+hC\alpha^{1-\e})\phi(\cdot, t_{i+1})
\end{equation}
 on
$B_{(1-\alpha^\e)r(t_{i+1})}(0)$. If we further assume
 \begin{equation} \label{phitilde}
 u(x,t) \leq (1+ C\alpha^{1-\e}) \tilde{\phi}(x,
t)
\end{equation}
 for $t_i \leq t \leq t_{i+1}$ and $x \in
B_{(1-\alpha^\e)r(t)}(0)$, then the condition (c) holds  with $C$
replaced by $hC$ at time $t=t_{i+1}$, i.e.,
\begin{equation} \label{lem55}
u(\cdot, t_{i+1}) \leq \phi(\cdot, t_{i+1}) +  L(hC+K) \alpha
r(t_{i+1})
\end{equation}
on $B_{r(t_{i+1})}(0) - B_{(1-\alpha^\e)r(t_{i+1})}(0)$.
\end{lemma}

\begin{proof}
For the proof of (\ref{lem5}), construct  caloric functions
$\psi_1$, $\psi_2$, $\psi_3$ and $\psi_4$ in $\Sigma:= \Omega_1 \cap
\{t_i < t < t_{i+1}\}$ with the following boundary values
$$
\begin{array}{ll}
&\left\{\begin{array} {ll} \psi_1 = \phi
 &\hbox{ on  }\,\,\,  B_{r(t_i)}(0) \times \{t=t_i\}\\ \\
\psi_1 =0 &\hbox{ otherwise on }\,\,  \partial \Sigma
\end{array}\right. \\ \\
 &\left\{\begin{array} {ll} \psi_2= u-\phi
&\hbox{ on }\,\,\, B_{(1-\alpha^\e)r(t_i)}(0)  \times \{t=t_i\} \\
\\
\psi_2=0 &\hbox{ otherwise on }\,\,
\partial \Sigma
\end{array}\right.\\ \\
&\left\{\begin{array} {ll} \psi_3= u-\phi &\hbox{ on  }\,\,\,
B_{r(t_i)}(0)- B_{(1-\alpha^\e)r(t_i)}(0) \times \{t=t_i\}  \\ \\
 \psi_3= 0 &\hbox{otherwise  on  }\,\, \partial \Sigma
\end{array}\right. \\ \\
 &\left\{\begin{array} {ll}  \psi_4= u &\hbox{ on  }\,\,\, \partial
\Sigma \cap \{t_i <t <t_{i+1}\}\\ \\
\psi_4 =0  &\hbox{ otherwise  on
}\,\, \partial \Sigma.
\end{array}\right.
\end{array}
$$
Then   in $\Sigma$
$$u = \psi_1 + \psi_2+\psi_3+\psi_4$$
where $\psi_1$ is radially symmetric since $\phi$ is radially
symmetric and $\Omega_1$ is radial. For $j=2,3,4$, let
$\psi_{j1}(\cdot)$ be the maximal radial function on
$B_{r(t_{i+1})}(0)$ such that $\psi_{j1}(\cdot) \leq \psi_j(\cdot,
t_{i+1})$ and write $\psi_j(\cdot, t_{i+1}) = \psi_{j1}(\cdot)+
\psi_{j2}(\cdot)$. Since $\phi(\cdot, t_{i+1})$ is the maximal
radial function $ \leq u(\cdot, t_{i+1})$,
$$\phi(\cdot, t_{i+1}) \geq \psi_1(\cdot, t_{i+1}) + \psi_{21}(\cdot) +\psi_{31}(\cdot)+\psi_{41}(\cdot). $$
Hence for (\ref{lem5}), it suffices to prove
$$\psi_{22}(\cdot) +\psi_{32}(\cdot)+\psi_{42}(\cdot) \leq hC\alpha^{1-\e}\phi(\cdot, t_{i+1}).$$

\vspace{0.1in}

1. First, we prove that for some constant $0<h_0 =h_0(n,M) <1$
\begin{equation} \label{355}
\psi_{22}(\cdot)  \leq h_0C\alpha^{1-\e}\phi(\cdot, t_{i+1}).
\end{equation}
Suppose that at some $x \in B_{r(t_{i+1})}(0)$
\begin{equation} \label{35}
\psi_2(x,t_{i+1}) >  \frac{1}{2} C \alpha^{1-\e} \psi_1(x, t_{i+1}).
\end{equation}
(Otherwise, (\ref{355}) would hold with $h_0=1/2$ since $\psi_1\leq
\phi$ and $\psi_{22} \leq \psi_2$.) Since $\Sigma$ is Lipschitz in a
parabolic scaling (Lemma~\ref{lemma2}), Lemma~\ref{ca} imply that
near the vanishing  boundary $\partial B_{r(t_{i+1})}(0)$,
$\psi_1(\cdot, t_{i+1})$ and $\psi_2(\cdot, t_{i+1})$ are comparable
to some harmonic functions vanishing on $\partial
B_{r(t_{i+1})}(0)$. Hence (\ref{35}) implies that on
$B_{r(t_{i+1})}(0)$
$$
\psi_2(\cdot,t_{i+1}) \geq h_1 C \alpha^{1-\e} \psi_1(\cdot,t_{i+1})
$$
 for some constant $0<h_1 =h_1(n,M)<1$.
Since $\psi_1$ is radially symmetric and $\psi_{21}$ is the maximal
radial function $\leq \psi_2$, the above inequality implies
\begin{equation} \label{366}
 \psi_{21} \geq h_1 C \alpha^{1-\e} \psi_1(\cdot,t_{i+1}).
\end{equation}
Let $h_0=1-h_1$, then on $B_{r(t_{i+1})}(0)$
\begin{eqnarray*}
\psi_{22}(\cdot) = \psi_2(\cdot,t_{i+1}) -\psi_{21}(\cdot) &\leq& C
\alpha^{1-\e} \psi_1(\cdot,t_{i+1}) -\psi_{21} (\cdot)\\  &\leq&
(1-h_1)C \alpha^{1-\e} \psi_1(\cdot,t_{i+1})
\\ &\leq& h_0C \alpha^{1-\e} \phi(\cdot,t_{i+1})
\end{eqnarray*}
where the first inequality follows from the assumption (a) with the
construction of $\psi_1$ and $\psi_2$, and the second inequality
follows from (\ref{366}). Hence we obtain the upper bound
(\ref{355}) of $\psi_{22}$.

\vspace{0.1in}

2. Next we show
\begin{equation} \label{356}
\psi_{32}(\cdot)  \leq \frac{1-h_0}{3}C\alpha^{1-\e}\phi(\cdot,
t_{i+1}).
\end{equation}
By the assumption (c) with the construction of $\psi_3$,
$$\max \psi_3(\cdot, t_{i})  \leq L(C+K) \alpha  r(t_i)$$
on the annulus $R:=B_{r(t_{i})}(0) - B_{(1-\alpha^\e)r(t_{i})}(0)$.
Let $|R|$ denote the volume of  $R$, then $|R| \approx \alpha^\e
|B_{r(t_i)}(0)|$. Hence there exists a small constant
$c(\alpha^\e,n)>0$ such that
\begin{equation} \label{Calpha}
\max \psi_3(\cdot, t_{i+1}) \leq c(\alpha^\e,n) L(C+K) \alpha  r(t_i)
\end{equation}
where we can observe $c(\alpha^\e,n) \rightarrow 0$ as $\alpha \rightarrow 0$.
Since $r(t_i) \approx r(t_{i+1}) \approx \max \psi_1(\cdot,
t_{i+1})$ (Lemma~\ref{lemma1}),
\begin{equation} \label{377}
\max \psi_3(\cdot, t_{i+1}) \leq c(\alpha^\e,n) C_0(n,M) L(C+K)
\alpha \max \psi_1(\cdot, t_{i+1})
\end{equation}
for some $C_0(n,M)>0$. Hence  on $B_{r(t_{i+1})}(0)$
\begin{eqnarray}
\psi_{32}(\cdot) \leq \psi_3(\cdot,t_{i+1}) &\leq& c(\alpha^\e,n)
C_1(n,M)L(C+K) \alpha  \psi_1(\cdot, t_{i+1}) \nonumber \\ &\leq&
c(\alpha^\e,n)
C_1(n,M)L(C+K) \alpha  \phi(\cdot, t_{i+1}) \nonumber \\
&\leq& c(\alpha^\e,n) C_2(n,M)L C \alpha^{\frac{\e+1}{2}}
\phi(\cdot,
t_{i+1}) \nonumber \\
&\leq& \frac{1-h_0}{3}C\alpha^{1-\e}\phi(\cdot, t_{i+1})
\label{newnew}
\end{eqnarray}
where $C_1(n,M)$ and $C_2(n,M)$ are positive constants depending on
$n$ and $M$, the first inequality follows from the almost
harmonicity of $ \psi_3(\cdot, t_{i+1})$ and $\psi_1(\cdot,
t_{i+1})$ with (\ref{377}), the second inequality follows since
$\psi_1 \leq \phi$,
 the third inequality follows from (\ref{K}), and the
 last inequality follows  if  $\alpha < \alpha(n,M)$ since $\e=2/3$ and $L$ and
$h_0$ are constants depending on $n$ and $M$.

\vspace{0.1in}

 3.  Since $\psi_{42}(\cdot)= \psi_4(\cdot, t_{i+1}) -\psi_{41}(\cdot)$
 where
$\psi_{41}(\cdot)$ is the maximal radial function $\leq
\psi_4(\cdot, t_{i+1})$,
\begin{eqnarray}
\max_{\partial B_s(0)} \psi_{42}(\cdot) &=& \max_{\partial
B_s(0)}\psi_4(\cdot, t_{i+1})- \psi_{41}|_{\partial B_s(0)}
\nonumber \\&=& \max_{\partial B_s(0)}\psi_4(\cdot, t_{i+1})-
\min_{\partial B_s(0) } \psi_4(\cdot, t_{i+1}) \label{4222}
\end{eqnarray}
for $0 \leq s\leq r(t_{i+1})$. Below we prove that the right hand
side of (\ref{4222}) is bounded from above by $$
\frac{1-h_0}{3}C\alpha^{1-\e}\phi(\cdot, t_{i+1})$$ if $0\leq s \leq
(1-\alpha^\e)r(t_{i+1})$.

Let $0 < s \leq (1-\alpha^\e)r(t_{i+1})$. Let $x_1$ be a  maximum
point of $\psi_4$ on $\partial B_s(0)$, i.e., $$ \max_{\partial
B_s(0) } \psi_4(\cdot, t_{i+1})  = \psi_4(x_1, t_{i+1}). $$ Since
$\psi_4$ is periodic in angle with period $\leq \alpha$, there
exists a minimum point $x_2$ of $\psi_4$  on $\partial B_s(0)$ such
that
$$ \min_{\partial B_s(0) } \psi_4(\cdot, t_{i+1})  =
\psi_4(x_2, t_{i+1})
$$ and $$
|x_1-x_2| < C_n\alpha r(t_{i+1})
$$ where $C_n$ is a dimensional constant. Recall that $\psi_4$ is a nonnegative
caloric function  in $\Sigma$ with
\begin{equation} \label{newnewnew}
\max_\Sigma \psi_4 = \max_{\partial \Omega_1 \cap \{t_i <
t<t_{i+1}\}} u \leq C(n,M) \alpha r(t_i) \leq C(n,M) K \alpha r(t_i)
\end{equation}
 where the first inequality
follows from (\ref{22222}) and the last inequality follows if
$K\geq1$. Now apply Lemma~\ref{Di262} for $\psi_4$ with $\delta =
\alpha^\e r(t_{i+1})$ and $\sigma =C_n \alpha r(t_{i+1})/\delta$.
Then the upper bound (\ref{newnewnew}) of $\psi_4$ implies that on
$Q^-_{\sigma \delta}(x_1, t_{i+1})$
$$ |\nabla \psi_4| \leq
\frac{ C(n,M)K \alpha r(t_i)}{ \alpha^\e r(t_{i+1})}.
$$  Since $|x_1-x_2| <C_n
\alpha r(t_{i+1}) = \sigma \delta$, the above bound on $| \nabla
\psi_4|$ yields that
\begin{eqnarray} \label{homogenization}
\max_{\partial B_s(0) } \psi_4(\cdot, t_{i+1}) - \min_{\partial
B_s(0) } \psi_4(\cdot, t_{i+1}) &=&
 \psi_4(x_1, t_{i+1}) -\psi_4(x_2, t_{i+1})
\nonumber \\
& \leq & C(n,M) K \alpha^{2-\e} r(t_{i}) \nonumber \\
& \leq & C_1 K \alpha^{2-\e} r(t_{i+1})
\end{eqnarray}
 for a constant $C_1$ depending on $n$ and $M$. Here, we can bound the
right hand side of (\ref{homogenization}) as below:
\begin{eqnarray*}
C_1K \alpha^{2-\e} r(t_{i+1}) &\leq& C_1 K \alpha^{2-2\e}
\min\{\phi(x, t_{i+1}): x \in B_{(1-\alpha^\e)r(t_{i+1})}(0)\} \\
&\leq& C_1 C\alpha^{3(1-\e)/2} \min \{\phi(x, t_{i+1}) : x \in
B_{(1-\alpha^\e)r(t_{i+1})}(0)\} \\&\leq &
\frac{1-h_0}{3}C\alpha^{1-\e}\min\{\phi(\cdot, t_{i+1}): x \in
B_{(1-\alpha^\e)r(t_{i+1})}(0)\}
\end{eqnarray*}
 where $C_1$'s are constants depending on $n$ and $M$, 
 the first inequality follows from (\ref{32}),
 the second inequality follows from (\ref{K}), and the
last inequality follows if $ \alpha <  \alpha(n,M)$ since $\e=2/3$ and $h_0$ is a constant depending on $n$ and $M$.
Combining the above inequality with (\ref{4222}) and
(\ref{homogenization})
 we obtain that
\begin{equation} \label{422}
\psi_{42}(\cdot, t_{i+1}) \leq
\frac{1-h_0}{3}C\alpha^{1-\e}\phi(\cdot, t_{i+1})
\end{equation}
on $B_{(1-\alpha^\e)r(t_{i+1})}(0)$.

 From the upper bounds (\ref{355}), (\ref{356}) and (\ref{422}) on
$\psi_{22}$, $ \psi_{32}$ and $\psi_{42}$, we conclude
$$\psi_{22} + \psi_{32} + \psi_{42} \leq \frac{h_0+ 2}{3} C\alpha^{1-\e}\phi(\cdot,
t_{i+1})$$ on $B_{(1-\alpha^\e)r(t_{i+1})}(0)$.  Hence
 the first
part of the lemma, that is  (\ref{lem5}), follows for the constant
$h := (h_0+ 2)/3<1$. \vspace{0.1in}

Next we prove the second part of the lemma, that is (\ref{lem55}).
Let $\Pi$ be a thin subregion of $ \Omega_1 \cap \{t_{i} \leq t\leq
t_{i+1}\}$ such that
$$\Pi_t= B_{r(t)}(0)-B_{(1-\alpha^\e)r(t)}(0) $$
for $t_{i} \leq t\leq t_{i+1}$.  Decompose $u$ into a sum of three
caloric functions $w_1$, $w_2$ and $w_3$, which are defined in
   $\Pi$ with
 the following boundary conditions
$$
\begin{array}{ll}
&\left\{\begin{array} {ll} w_1(\cdot, t) = u (\cdot, t)
 &\hbox{ on  } \partial B_{(1-\alpha^\e)r(t)}(0) \hbox{ for } t_{i} \leq t\leq t_{i+1}
 \\ \\
w_1 =0 &\hbox{ otherwise on }  \partial \Pi
\end{array}\right. \\ \\
&\left\{\begin{array} {ll} w_2(\cdot, t) = u (\cdot, t)
 &\hbox{ on  } \partial B_{r(t)}(0) \hbox{ for } t_{i} \leq t\leq t_{i+1}
 \\ \\
w_2 =0 &\hbox{ otherwise on }  \partial \Pi
\end{array}\right.\\ \\
&\left\{\begin{array} {ll} w_3(\cdot, t_{i})=  u (\cdot, t_{i})
 &\hbox{ on  } B_{r(t_{i})}(0)-B_{(1-\alpha^\e)r(t_{i})}(0)  \\ \\
w_3 =0 &\hbox{ otherwise on }  \partial \Pi.
\end{array}\right.
\end{array}
$$
Observe  $u = w_1 + w_2 + w_3$ in $\Pi$. Let $w_{11}$ be the maximal
radial function such that $w_{11}(\cdot) \leq w_1(\cdot, t_{i+1})$
and let $w_{31}$ be the maximal radial function such that
$w_{31}(\cdot) \leq w_3(\cdot, t_{i+1})$. Then on
$B_{r(t_{i+1})}(0)-B_{(1-\alpha^\e)r(t_{i+1})}(0) \times
\{t=t_{i+1}\}$
$$u-\phi \leq (w_1-w_{11}) + w_2 +
(w_3-w_{31})$$ since $\phi(\cdot, t_{i+1})$ is the maximal radial
function $\leq u(\cdot, t_{i+1})$. Hence for (\ref{lem55}), it
suffices to prove that the right hand side of the above inequality
is bounded by $L(hC+K) \alpha r(t_{i+1})$.

\vspace{0.1in}

1. Bound on $w_1-w_{11}$:
\begin{eqnarray*}
w_1 (\cdot, t_{i+1}) - w_{11}(\cdot) &\leq& C \alpha^{1-\e}
w_{11}(\cdot) \\&\leq&  C \alpha^{1-\e}  C_1(n,M)\alpha^\e
r(t_{i+1})
\\ &\leq&
\frac{h C L}{2}   \alpha r(t_{i+1})
\end{eqnarray*}
 where the first inequality follows
from the assumption (\ref{phitilde}) since $$w_1=u \leq \tilde{\phi}
+ C \alpha^{1-\e}\tilde{\phi}
$$ on  $ \partial B_{(1-\alpha^\e)r(t)}(0) \times \{t\}$ and
$w_{11}$ is bounded below by a radial caloric function in $\Pi$ with
boundary value $\tilde{\phi}$ on $ \partial B_{(1-\alpha^\e)r(t)}(0)
\times \{t\}$,
 the second inequality follows from $\max
w_{11} \leq \displaystyle{\max_\Pi} u \leq C_1(n,M) \alpha^\e
r(t_{i+1})$, and the last inequality holds if $L \geq L(n,M)$.

\vspace{0.1in}

2. Bound on $w_2$: For $t_{i} \leq t \leq t_{i+1}$,
\begin{equation} \label{new}
\max_{\partial B_{r(t)}(0)} u(\cdot, t) \leq C(n,M) \alpha r(t) \leq
C(n,M) K\alpha r(t)
\end{equation}
where the first inequality follows from  (\ref{22222}) and the last
inequality follows if $K\geq 1$.
 Hence from the construction of
$w_2$,
\begin{eqnarray*}
\max w_2(\cdot, t_{i+1}) &\leq& C(n,M)K\alpha r(t_{i+1})  \\&\leq&
\frac{ L K }{2} \alpha r(t_{i+1})
\end{eqnarray*}
where  the last inequality holds if $L \geq L(n,M)$.

\vspace{0.1in}

 3. Bound on $w_3-w_{31}$:  A similar argument as in (\ref{377}) shows that
(c) implies
$$w_3(\cdot, t_{i+1})-w_{31}(\cdot)  \leq c(\alpha^\e,n) C(n,M) L(C+K)
\alpha r(t_{i+1})$$ where   $c(\alpha^\e,n)$ is a constant given as in (\ref{Calpha}). Then for     $\alpha<\alpha(n,M)$, 
$$w_3(\cdot, t_{i+1})-w_{31}(\cdot) \leq \frac{hL(C+K)}{2}  \alpha
r(t_{i+1}).$$ Combing the above bounds on $w_1-w_{11}$, $w_2$, and
$w_3-w_{31}$, we conclude
\begin{eqnarray*}
u(\cdot, t_{i+1})-\phi (\cdot, t_{i+1})& \leq& w_1(\cdot,
t_{i+1})-w_{11} + w_2(\cdot, t_{i+1}) + w_3(\cdot, t_{i+1})-w_{31}
\\&\leq&  L(hC+K) \alpha r(t_{i+1})
\end{eqnarray*}
\end{proof}

\noindent {\bf Remark 6.} In Lemma~\ref{lemma5}, we assume $K\geq 1$  for
simplicity of the proof. In fact, this assumption is used only in the proofs of (\ref{newnewnew}) and (\ref{new}), where we find a bound on $\max_{\partial
 \Omega_1}u$. Later in the paper, we modify and improve the inner
 region $\Omega_1$ so that  (\ref{newnewnew}) and (\ref{new}) are
 guaranteed also for $K < 1$ (see Corollary~\ref{cor2}). Hence, we will be able to iterate Lemma~\ref{lemma5} and Lemma~\ref{lemma6} for a
 decreasing sequence of $K<1$.

\begin{corollary} \label{cor1}
Let $0<h<1$ and $\e=2/3$ be as in Lemma~\ref{lemma5}.  Let $m$ be
the largest integer satisfying
$$ 1 < \alpha^{\frac{\e-1}{2}}h^{m}. $$
 Then  for $j \geq m+2$,  (a) and (c) of Lemma~\ref{lemma5} hold with $C$ replaced by
$\alpha^{\frac{1-\e}{2}}C$ at time $t=t_j$.  Here  $C>0$ is a
constant depending on $n$ and $M$.
\end{corollary}
\begin{proof}
 By Lemma~\ref{lemma3} and Lemma~\ref{lemma4}, the
conditions (a) and (b) of Lemma~\ref{lemma5} are satisfied with
$C=K=C(n,M) \geq 1$ for $i \geq 2$. The condition (c) of
Lemma~\ref{lemma5} follows from  (\ref{refrefr}) for $i \geq 2$.
Also the condition (\ref{phitilde}) holds for $t \geq t_2$ by
Lemma~\ref{lemma4} with $\phi \leq \tilde{\phi}$. Hence applying
Lemma~\ref{lemma5}, we obtain that the conditions (a) and (c) hold
with $C$ replaced by $hC$ for $i \geq 3$.

On the other hand, the inequality (\ref{K}) of the condition (b)
holds for the constants $K$ and $hC$, since  $K=C<
\alpha^{\frac{\e-1}{2}}hC$ for $\alpha< \alpha(n,M)$.  Hence the condition (b) holds for $i
\geq 3$ with the improved constant $hC$.

Next we verify the condition (\ref{phitilde}) with $C$ replaced by
$hC$ for $t \geq t_3$. Fix  $\tau \in (t_{i}, t_{i+1})$ for $i \geq
3$. Decompose the time interval $(0,T)$ so that $$0<
s_1=t_1<...<s_{i-1}=t_{i-1} <s_i=\tau <s_{i+1} =t_{i+2}<...<T.$$
Then by a similar argument as in the proof of (\ref{lem5}) of
Lemma~\ref{lemma5},
$$
u (\cdot, \tau) \leq (1+ hC \alpha^{1-\e})\tilde{\phi}(\cdot, \tau)
\,\,\hbox{ on }\,\,B_{(1-\alpha^\e)r(\tau)}(0).$$ Hence the
condition (\ref{phitilde}) is satisfied for  $t\geq t_3$ with $C$
replaced by $hC$. Then applying Lemma~\ref{lemma5} for $i\geq 3$, we
obtain a better constant $h^2C$ for the conditions (a) and (c) for
$i \geq 4$.

 Now
let $m$ be the largest integer satisfying
$$ 1 < \alpha^{\frac{\e-1}{2}}h^{m}. $$
Then the inequality (\ref{K}) holds with $C$ replaces by $h^jC$ for
$1 \leq j\leq m$. Hence we can iterate Lemma~\ref{lemma5}  m times,
as above, and obtain  the improved constant
$\alpha^{\frac{1-\e}{2}}C$ in (a) and (c) for $i \geq m+2$. In other
words, for $i \geq m+2$
$$
u (\cdot, t_i) \leq (1+ \alpha^{\frac{1-\e}{2}}C
\alpha^{1-\e})\phi(\cdot, t_i)
$$
 on
$B_{(1-\alpha^\e)r(t_i)}(0)$ and
$$
u(\cdot, t_i) \leq \phi(\cdot, t_i) +  L( \alpha^{\frac{1-\e}{2}}C
+K) \alpha r(t_i)
$$
on $B_{r(t_i)}(0) - B_{(1-\alpha^\e)r(t_i)}(0)$.
\end{proof}

\section{Improvement on flatness by  interior estimate;\\
Asymptotic behavior of free boundary} \label{sec7} In this section
we show that the improved interior estimate, as in
Corollary~\ref{cor1}, propagates to the free boundary at later times
and gives an improved estimate on the location of the free boundary.
More precisely, we improve the constant $K$ in condition (b) of
Lemma~\ref{lemma5}, using the improved constants in the conditions
(a) and (c).

\begin{lemma} \label {lemma6}
Suppose that (a), (b) and (c) of Lemma~\ref{lemma5} hold for $i\geq
i_0$ with $C=\beta$, a small constant. Then for $i \geq i_0+2$, (b)
and (c) holds with $K$ replaced by $C_1\beta$, for a constant
$C_1>0$  depending on $n$ and $M$. In other words,   $\Gamma_t(u)$
is contained in $B_{(1+ C_1 \beta \alpha)r(t)}(0) -B_{r(t)}(0)$ for
$ t \geq t_{i_0+2}$, and
$$u(\cdot, t_{i}) \leq \phi(\cdot,
t_{i}) +  L(\beta+ C_1\beta) \alpha r(t_{i})$$
 on $B_{r(t_{i})}(0) - B_{(1-\alpha^\e)r(t_{i})}(0)$ for $i \geq i_0+2$.
\end{lemma}
\begin{proof}
To prove the lemma at $t=t_{i_0+2}$, we will construct a radially
symmetric caloric function $w\leq u $ and a radially symmetric
supercaloric function $v \geq u$ such that  their free boundaries
$\Gamma_{t_{i_0+2}}(w)$ and $\Gamma_{t_{i_0+2}}(v)$ are located in
the  $C_1\beta \alpha r(t_{i_0+2})$-neighborhood of each other.
Recall that $u$ is well approximated by a radial function $\phi$ on
each dyadic time interval. However the function $\phi$ (or $w_1$
which will be constructed below) does not catch up the change in
values of $u$ caused by the displacement of the free boundary from
$\partial \Omega_1$. Hence we modify the approximating function
$w_1$ by adding an auxiliary function $w_2$, and construct two
caloric functions $w\leq u$ and $v \geq u$ using the modified
approximation, $w_1+w_2$, of $u$.

Let $w_1$ solve
$$
\left\{ \begin{array}{lll} \Delta w_1=\partial w_1 /\partial t
&\hbox{ in }&\Omega_1 \cap \{t_{i_0} < t< t_{i_0+2}\}\\ \\
w_1=\phi &\hbox{ on }& \Omega_1 \cap \{t=t_{i_0}\} \\ \\
 w_1=0 &\hbox{ on }& \partial \Omega_1 \cap \{t_{i_0} < t< t_{i_0+2}\}.
\end{array} \right.
 $$
Note that $w_1 = \phi$ for $t_{i_0} \leq t < t_{i_0+1}$ and $w_1
\leq \phi$ for $t_{i_0+1} \leq t \leq t_{i_0+2}$. Recall that
$\psi_i$ ($1 \leq i \leq 4$) are the caloric functions constructed
in the proof of Lemma~\ref{lemma5}.
 Let $\tilde{\psi}_4$ solve
$$
\left\{ \begin{array}{lll} \Delta \tilde{\psi}_4=\partial
\tilde{\psi}_4/\partial t &\hbox{ in }&\Omega_1 \cap \{t_{i_0} < t<
t_{i_0+2}\}\\ \\
\tilde{\psi}_4=0 &\hbox{ on }& \Omega_1 \cap \{t=t_{i_0}\} \\ \\
 \tilde{\psi}_4=u &\hbox{ on }& \partial \Omega_1 \cap \{t_{i_0} < t< t_{i_0+2}\}.
\end{array} \right.
 $$
Note that  $\tilde{\psi}_4 = \psi_4$ for $t_{i_0} \leq t <
t_{i_0+1}$ and $\tilde{\psi}_4 \geq \psi_4$ for $t_{i_0+1} \leq t
\leq t_{i_0+2}$.    Let $\Sigma' $ be a space time region in $
\Omega_1 \cap \{t_{i_0} \leq t \leq t_{i_0+2}\}$ such that its time
cross-section
$$\Sigma'_t = B_{r(t)}(0) - B_{(1-\alpha^\e)r(t)}(0)
$$ for
$t_{i_0}\leq t \leq t_{i_0+2}$. Let $w_2$ solve
$$
\left\{ \begin{array}{lll} \Delta w_2=\partial w_2 /\partial t
&\hbox{ in }& \Sigma' \\ \\
w_2=0 &\hbox{ on }&   \{t=t_{i_0}\} \cup \partial \Omega_1
\\ \\
w_2= \tilde{\psi}_4 &\hbox{ on }& \partial \Sigma' -
\{t=t_{i_0}\} -
\partial \Omega_1.
\end{array} \right.
 $$
Note that $w_2$ has a nonzero boundary values $\tilde{\psi}_4(\cdot,
t)$ only on the inner boundary $\partial B_{(1-\alpha^\e)r(t)}(0)
\times \{t\}$ of $\Sigma'$. Now let
$$w=w_1+w_2 \hbox{ in } \Sigma'.$$
Then  $w =\phi \leq u$ on $\{t=t_{i_0}\}$ and $w = w_1 +
\tilde{\psi}_4 \leq u$ on $\partial
B_{(1-\alpha^\e)r(t)}(0)\times\{t\}$ since $w_1 + \tilde{\psi}_4$ is
caloric  in $\Omega_1\cap \{t_{i_0} < t< t_{i_0+2}\}$ with $w_1 +
\tilde{\psi}_4 \leq u$ on $\partial \Omega_1 \cup \{t=t_{i_0}\}$.
Hence by comparison
$$w \leq u \hbox{ in } \Sigma'.$$

Next to construct a supercaloric function $v \geq u$, we modify the
boundary of $w$ on the time interval $[t_{i_0+1}, t_{i_0+2}]$ and
also modify the values of $w$ in the new region so that it is a
supersolution of ($P$) with larger boundary values than $u$. Let
$f(t)$ be the linear function defined on the interval $[t_{i_0+1},
t_{i_0+2}]$ such that
$$
\left\{ \begin{array}{l} f(t_{i_0+1}) =1 - C_1K \alpha\\ \\
f(t_{i_0+2} )=1 - 2C_1 \beta \alpha \end{array} \right.
$$
 where $C_1=C_1(n,M)$ is a sufficiently large constant which will be determined
 later. Here we assume $K>2\beta$ since otherwise the lemma
 would hold with $C_1=2$.
 For a fixed $t \in [t_{i_0+1},
t_{i_0+2}]$, let $g(x,t)$ be the harmonic function defined in $B_{
r(t)/f(t) }(0) -B_{(1-\alpha^\e) r(t)/f(t)}(0)$ such that
$$
\left\{\begin{array} {lll} g(x,t)= 1 &\hbox{ for }& x\in \partial
B_{ r(t)/f(t) }(0) \\ \\
 g(x,t)= 1- C_1K \alpha  &\hbox{ for }& x
\in
\partial B_{(1-\alpha^\e) r(t)/f(t) }(0).
\end{array}\right.
$$
Let $\Pi$ be a space time region constructed on the time interval
$[t_{i_0+1}, t_{i_0+2}]$ such that its time cross-section
$$\Pi_t  = B_{ r(t)/f(t)}(0) - B_{(1-\alpha^\e) r(t)/f(t)   }$$
for $t\in[t_{i_0+1}, t_{t+2}]$. Now construct a function  $v$ in
$\Pi$ as follows
$$v(x,t) = g(x,t)w(f(t)x,t).$$
We will show that $v$ is a supersolution of ($P$) satisfying  $v
\geq u$ on the parabolic boundary of $\Pi$.

\vspace{0.1in}

1. To prove that $v$ is supercaloric in $\Pi$, we find some bounds
on $|f'(t)|$, $|g_t|$, $|\nabla g|$, $w$, $|\nabla w|$ and  $|w_t|$.
\begin{itemize}
\item[(1)] Since $f$ is linear and $t_{i_0+2}-t_{i_0+1} \approx r^2(t_{i_0+1})$ (Lemma~\ref{lemma1})
\begin{equation} \label{f'}
 |f'(t)| \leq \frac{C(n,M)C_1K \alpha}{r^2(t_{i_0+1})}.
 \end{equation}

\item[(2)] Since $g(\cdot, t)$ is harmonic  on the annulus
$B_{ r(t)/f(t) }(0) -B_{(1-\alpha^\e) r(t)/f(t) }(0)$
\begin{equation} \label{nablag}
\frac{c(n)C_1 K \alpha}{\alpha^e r(t)} \leq |\nabla g| \leq
\frac{C(n)C_1 K \alpha}{\alpha^e r(t)}.
\end{equation}

\item[(3)] From the construction of $g$
\begin{eqnarray} \label{gtt}
|g_t| &\leq& \max |\nabla g| ( \frac{d}{dt}\frac{r(t)}{f(t)}+ r'(t))
\nonumber\\ &\leq& \max |\nabla g|C(n,M)(\frac{C_1K \alpha}{r(t_{i_0+1})} + r'(t)) \nonumber \\
&\leq &\max |\nabla g|C(n,M)\frac{C_1K \alpha+1}{r(t_{i_0+1})}
\nonumber \\
&\leq& \frac{C(n,M)C_1 K \alpha(C_1K \alpha+1) }{ \alpha^e
r^2(t_{i_0+1})}
\end{eqnarray}
where the second inequality follows from (\ref{f'}),  the third
inequality follows from the Lipschitz property of $\Omega_1$
(Lemma~\ref{lemma2}) with Lemma~\ref{lemma1}, and the last
inequality follows from (\ref{nablag}) with Lemma~\ref{lemma1}.

\item[(4)]
Since $\max \phi(\cdot, t) \approx \max u(\cdot, t) \approx r(t) $,
\begin{equation} \label{maxwc}
\max w_1(\cdot, t) \approx r(t).
\end{equation}

\item[(5)] Since $w_1$ is a caloric function vanishing on the Lipschitz boundary $\partial \Omega_1
\cap \{ t_{i_0}<t<t_{i_0+2}\}$,  $w_1$ is almost harmonic near
$\partial \Omega_1 \cap \{ t_{i_0+1}\leq t \leq t_{i_0+2}\}$ by
Lemma~\ref{ca}. Hence (\ref{maxwc}) implies that for $(x,t) \in
\Sigma' \cap\{t_{i_0+1} \leq t\leq t_{i_0+2}\}$
\begin{equation} \label{wxtwxt}
c(n,M){\rm dist} (x, \partial B_{r(t)}(0)) \leq w_1(x,t) \leq
C(n,M){\rm dist} (x, \partial B_{r(t)}(0)).
\end{equation}

\item[(6)] Applying Lemma~\ref{caf421}
to the re-scaled $w_1(\sqrt{t_{i_0+2}}x,  t_{i_0+2} t)$, we  obtain
that for $t \in (t_{i_0+1}, t_{i_0+2})$ and $x \in B_{r(t)}(0)
-B_{(1-\alpha^\e)r(t)}(0)$
\begin{equation} \label{wxtwxt2}
c(n,M) \leq |\nabla w_1 (x,t)| \leq C(n,M)
\end{equation}
and
\begin{equation} \label{wtwtwt}
|\partial w_1/\partial t| \leq \frac{C(n,M)r(t_{i_0+1})}{t_{i_0+2}}
\leq \frac{C(n,M)}{r(t_{i_0+1})}
\end{equation}
where (\ref{wxtwxt2}) and the first inequality of (\ref{wtwtwt})
follow from (\ref{wxtwxt}) and the second inequality of
(\ref{wtwtwt}) follows from Lemma~\ref{lemma1}.

\item[(7)] From the construction of $\tilde{\psi}_4$,
$$ \max \tilde{\psi}_4 = \max_{\partial \Omega_1 \cap \{t_{i_0} \leq t
\leq  t_{i_0+2}\} } u \leq C(n,M) K \alpha r(t_{i_0}) $$ where the
last inequality follows from $|\nabla u| \leq C_0M$
(Lemma~\ref{cak}) and the condition (b). Hence by Lemma~\ref{Di262},
$$\max_{\partial B_{(1-\alpha^\e)r(t)}(0) } |\nabla \tilde{\psi}_4(x,t)| \leq
C(n,M)  K \alpha^{1-\e}$$ and
$$ \max_{\partial B_{(1-\alpha^\e)r(t)}(0) }|\frac{\partial \tilde{\psi}_4
}{\partial t}(x,t)| \leq \frac{C(n,M)  K \alpha^{1-2\e}}{r(t)}.
$$
Recall that the caloric function $w_2$ vanishes on the Lipschitz
boundary $\partial \Omega_1$ and on $\{t=t_{i_0}\}$, and it has
nonzero boundary values $\tilde{\psi}_4$ only on the inner boundary
of $\Sigma'$, i.e., on $\partial B_{(1-\alpha^\e)r(t)}(0) \times
\{t\}$. Since the inner boundary of $\Sigma'$ is also Lipschitz in a
parabolic scaling, the above bounds on $|\nabla \tilde{\psi}_4|$ and
$|\partial \tilde{\psi}_4 /\partial t|$  yield that
\begin{equation} \label{wwwww}
|\nabla w_2| \leq C(n,M)  K \alpha^{1-\e}, \,\,\, |\partial w_2
/\partial t| \leq \frac{C(n,M)  K \alpha^{1-2\e}}{r(t)} \hbox{ in }
\Sigma'.
\end{equation}

\item[(8)] Combining  (\ref{wxtwxt}), (\ref{wxtwxt2}),
(\ref{wtwtwt}) and (\ref{wwwww}), we obtain
\begin{equation} \label{wxtwxt'}
c(n,M){\rm dist} (x, \partial B_{r(t)}(0)) \leq w(x,t) \leq
C(n,M){\rm dist} (x, \partial B_{r(t)}(0))
\end{equation}
\begin{equation} \label{wxtwxt2'}
c(n,M) \leq |\nabla w (x,t)| \leq C(n,M)
\end{equation}
and
\begin{equation} \label{wtwtwt'}
|\partial w/\partial t| \leq \frac{C(n,M)}{r(t_{i_0+1})}.
\end{equation}

\end{itemize}

Next we prove that $v$ is supercaloric in $\Pi$.
\begin{eqnarray*}
\Delta v -v_t
&\leq&
2f\nabla g\cdot \nabla w +g f^2 \Delta w -g_t w+ g|\nabla w||f'||x|-gw_t  \\
&\leq& 2f\nabla g\cdot \nabla w -g_t w+ g|\nabla w||f'||x|+2C_1K
\alpha g|w_t|
\\ &\leq&
 -C(n,M)|\nabla g| |\nabla w| +|g_t|w + g|\nabla w||f'||x|+2C_1K \alpha g|w_t|
\\
&\leq& \frac{-C(n,M) C_1 K \alpha}{\alpha^e r(t_{i_0+1})} +
\frac{C'(n,M) C_1 K \alpha}{ r(t_{i_0+1})} \leq 0
\end{eqnarray*}
where  the first and second inequalities follow from the
construction of $v$, the third inequality follows from the
monotonicity of $w_1$, i.e., from Lemma~\ref{cafmonoton} applied for
$w_1$  with the gradient bounds (\ref{wxtwxt2}) and (\ref{wwwww}) of
$w_1$ and $w_2$ (note $g$ is radial and increasing in $|x|$),
 the fourth inequality follows from (\ref{f'}), (\ref{nablag}),
(\ref{gtt}), (\ref{wxtwxt'}), (\ref{wxtwxt2'})  and (\ref{wtwtwt'})
for constants $C(n,M)$ and $C'(n,M)$ depending on $n$ and $M$, and
the last inequality follows if $\alpha< \alpha(n,M)$. 

\vspace{0.1 in}

2. For $x\in \partial B_{ r(t)/f(t) }(0)$ and  $t \in [t_{i_0+1},
t_{i_0+2}]$,
\begin{eqnarray*}
|\nabla v(x,t)| &=&|w(f(t)x,t)\nabla g(x,t) + g(x,t) f(t) \nabla
w(f(t)x,t)| \\  &=&|g(x,t) f(t) \nabla w(f(t)x,t)| \\  &\leq
&(1-2C_1\beta \alpha)|\nabla w| \leq 1
\end{eqnarray*}
where the second equality follows since $w=w_1+w_2=0$ on $\partial
B_{ r(t) }(0)$, the first inequality follows since $f \leq
1-2C_1\beta \alpha$ and $g = 1$ on $\partial B_{ r(t)/f(t) }(0)$,
and the last inequality follows since $w \leq u$ and $\partial
\Omega_1$ and $\partial \Omega(u)$  intersect on each small time
interval. Hence $v$ is a supersolution of ($P$) in $\Pi$.

\vspace{0.1 in}

 3. We show $u \leq v$ on  $\Pi \cap
\{t=t_{i_0+1}\}$. Recall that $w_1 =\phi$ for $t_{i_0} \leq t <
t_{i_0+1}$ and $w_1$ is not necessarily equal to $\phi$  at time
$t=t_{i_0+1}$ since $\phi(\cdot, t_{i_0+1})$ is defined to be the
maximal radial function $\leq u(\cdot, t_{i_0+1})$. However by a
similar argument as in the proof of Lemma~\ref{lemma5}, we can show
that if the assumptions (a), (b) and (c) of Lemma~\ref{lemma5}  hold
for $i=i_0$ and $C=\beta$ then
\begin{equation} \label{wassumption}
u(\cdot, t_{i_0+1}) \leq w_1(\cdot, t_{i_0+1}) + C(n,M)  K \alpha
r(t_{i_0+1})
\end{equation}
on $B_{r(t_{i_0+1})}(0) - B_{(1-\alpha^\e)r(t_{i_0+1})}(0)$ where
$C(n,M)$ is a constant depending on $n$ and $M$. To prove
(\ref{wassumption}), recall that
$$u= \psi_1+ \psi_2 +\psi_3+\psi_4$$
in $\Omega_1 \cap \{t_{i_0} \leq t \leq t_{i_0+1}\}$, where $\psi_i$
are the caloric functions constructed in the proof of
Lemma~\ref{lemma5} with $i=i_0$. From the construction of $\psi_1$,
we can observe $w_1(\cdot,t_{i_0+1})=\psi_1(\cdot,t_{i_0+1})$. Also
on $B_{r(t_{i_0+1})}(0) \times \{t=t_{i_0+1}\}$
\begin{eqnarray}
\psi_2 + \psi_3 + \psi_4 &\leq& \beta \alpha^{1-\e}\psi_1 + \psi_3 +
\psi_4 \nonumber \\
&\leq& 2\beta \alpha^{1-\e}\psi_1  + \psi_4 \nonumber \\
&\leq& 2\beta \alpha^{1-\e}\psi_1  + C(n,M)K \alpha r(t_{i_0+1})
\label{betabeta}
\end{eqnarray}
where the first inequality follows from the construction of $\psi_1$
and $\psi_2$ and the condition (a) with
 $i=i_0$ and $C=\beta$, the second inequality follows from (\ref{newnew}) with $C=\beta$
 and with
 $\psi_1$ in place of $\phi$, and the
last inequality follows from (\ref{newnewnew}) and
Lemma~\ref{lemma1}. Hence  on $B_{r(t_{i_0+1})}(0) -
B_{(1-\alpha^\e)r(t_{i_0+1})}(0) \times \{t=t_{i_0+1}\}$
\begin{eqnarray*}
u &=& w_1 + \psi_2 + \psi_3 +\psi_4
\\&\leq& (1+2 \beta \alpha^{1-\e})w_1  + C(n,M)K \alpha r(t_{i_0+1})
\\&\leq& w_1 + C(n,M) (\beta \alpha r(t_{i_0+1})  +
K \alpha r(t_{i_0+1})) \\
&\leq& w_1 + C(n,M)  K \alpha r(t_{i_0+1})
\end{eqnarray*}
where the equality follows from $\psi_1 =w_1$, the first inequality
follows from (\ref{betabeta}), the second inequality follows from
(\ref{wxtwxt2}), and the last inequality follows since $K>2\beta$.
Hence we obtain (\ref{wassumption}).

Now on  $\Pi \cap \{t=t_{i_0+1}\}$ ($=B_{r(t_{i_0+1})}(0) -
B_{(1-\alpha^\e)r(t_{i_0+1})}(0) \times \{t=t_{i_0+1}\}$)
\begin{eqnarray*}
v(x,t_{i_0+1}) &\geq &  (1-C_1K \alpha) w_1 ((1-C_1 K \alpha)x,
t_{i_0+1})\\
&\geq& (1-C_1K \alpha) (w_1 (x, t_{i_0+1}) + c(n,M)C_1 K \alpha
r(t_{i_0+1})) \\&\geq& w_1 (x, t_{i_0+1}) -C(n,M)C_1K
\alpha^{1+\e}r(t_{i_0+1}) + C(n,M)C_1 K \alpha r(t_{i_0+1})
\\&\geq& w_1 (x, t_{i_0+1}) + C(n,M)C_1 K \alpha r(t_{i_0+1})
 \\ &\geq & u(x, t_{i_0+1})
\end{eqnarray*}
where the first inequality follows from the construction of $v$, the
second inequality follows from (\ref{wxtwxt2}) with the monotonicity
of $w_1$ (Lemma~\ref{cafmonoton}), the third inequality follows from
(\ref{wxtwxt2}), the forth inequality follows if $\alpha<
\alpha(n,M)$, and the last
inequality follows form (\ref{wassumption}) if we choose a large constant $C_1$
 depending on $n$ and $M$.

\vspace{0.1in}

4. We show $v\geq u$ on the inner lateral boundary of $\Pi$, i.e.,
on the set $\partial B_{(1-\alpha^\e)r(t)}(0) \times \{t\}$ for $t
\in [t_{i_0+1}, t_{i_0+2}]$. By the construction of $w$,
\begin{equation} \label{tto}
w =w_1+\tilde{\psi}_4 \hbox{ on }\partial
B_{(1-\alpha^\e)r(t)}(0)\times \{t\}.
\end{equation}
 Here recall that
$w_1+\tilde{\psi}_4$ is a caloric function in $\Omega_1 \cap
\{t_{i_0} \leq t \leq t_{i_0+2}\}$ with boundary values $u$ on
$\partial \Omega_1$, and $\phi$ on $\{t=t_{i_0}\}$. Then by a
similar argument as in (\ref{betabeta}),
\begin{equation} \label{ttoo}
u - ( w_1+\tilde{\psi}_4   ) \leq 2\beta \alpha^{1-\e}w_1
\end{equation}
in $\Omega_1 \cap \{t_{i_0+1} \leq t\leq t_{i_0+2}\}$. Combining
(\ref{tto}) and (\ref{ttoo}) we obtain that on $\partial
B_{(1-\alpha^\e)r(t)}(0) \times \{t\}$ and for $t \in [t_{i_0+1},
 t_{i_0+2}]$
\begin{equation} \label{betabeta'}
u \leq w + 2\beta \alpha^{1-\e}w_1 \leq (1+2\beta \alpha^{1-\e})w.
\end{equation}
 Now for $x\in
\partial B_{(1-\alpha^\e)r(t)}(0)$ and $t \in [t_{i_0+1},
t_{i_0+2}]$,
\begin{eqnarray*}
v(x,t) &\geq&  (1-C_1K \alpha) w ((1-2C_1 \beta \alpha)x, t)\\
&\geq& (1-C_1K \alpha) (w (x, t) + C(n,M)C_1 \beta \alpha r(t))
\\&\geq& w (x, t) + C(n,M)C_1 \beta \alpha r(t)-C(n,M)C_1K
\alpha^{1+\e}r(t)
\\&\geq&  w(x, t) + C(n,M)C_1 \beta \alpha r(t)-C(n,M)C_1
\alpha^{\frac{\e-1}{2}} \beta \alpha^{1+\e}r(t)
\\&\geq& w(x, t) + C(n,M)C_1 \beta \alpha r(t)
\\&\geq&(1 + 2\beta \alpha^{1-\e})w(x, t)
 \\ &\geq & u(x, t)
\end{eqnarray*}
where the first inequality follows since $w$ is decreasing in $|x|$,
the second and third inequalities follow from (\ref{wxtwxt2'}), the
fourth inequality follows from  the assumption (\ref{K}) with
$C=\beta$, that is $K < \alpha^{(\e-1)/2}\beta$, the fifth
inequality follows since $\e =2/3$, the sixth inequality follows
from (\ref{wxtwxt2'}) if $C_1=C_1(n,M)$ is chosen 
large, and the last inequality follows from (\ref{betabeta'}).

\vspace{0.1in}

5. Conclude from 1, 2, 3 and 4 that  $v$ is a supersolution of ($P$)
in $\Pi$ such that  $v \geq u$ on  $\partial \Pi \cap
\{t=t_{i_0+1}\}$ and on the inner lateral boundary of $\Pi$. By
comparison, $v \geq u$ in $\Pi$. Recall that $w \leq u$ in
$\Sigma'$. Hence the free boundary $\Gamma_t(u)$ of $u$ is trapped
between $\Gamma_t(v)$ and $\Gamma_t(w)$ for $t_{i_0+1} \leq t \leq
t_{i_0+2}$. Now let $d(t)$ be the distance between $\Gamma_t(v)$ and
$\Gamma_t(w)$. Then by
 the construction of $v$,
$$d(t)=r(t)(\displaystyle{\frac{1}{f(t)}}-1).$$
 Since $0<f(t)<1$ increases in time on the time interval $[t_{i_0+1}, t_{i_0+2} ]$,
  the function   $1/f(t)-1$ decreases in time on $[t_{i_0+1}, t_{i_0+2} ]$.
  Hence we can obtain an improved
   estimate on the location of the free boundary at the later time
   $t=t_{i_0+2}$. Since $f(t_{i_0+2})= 1 -2C_1\beta \alpha$,
$$d(t_{i_0+2}) \leq 3C_1 \beta \alpha r(t_{i_0+2}).$$
We conclude that the condition (b) holds with $K$ replaced by $3C_1
\beta$ for a constant $C_1$ depending on $n$ and $M$.

\vspace{0.1in}

6. Lastly, if (b) holds for $K=3C_1 \beta$ and $i \geq i_0+2$, i.e.,
$$\Gamma_t(u)  \subset B_{(1+3C_1 \beta \alpha)r(t)}(0)
-B_{r(t)}(0)$$ for $t \geq t_{i_0+2}$, then since $|\nabla u| \leq
C_0M$,
$$ u -\phi = u \leq C(n,M) \beta \alpha r(t)$$ on $\partial \Omega_1
\cap \{t \geq t_{i_0+2}\}$.
\end{proof}

\noindent {\bf Remark 7.} Note that  in the proof of Lemma~\ref{lemma6}, we
use the condition $K \geq 1$ only for (\ref{betabeta}), i.e., for
(\ref{newnewnew}).

\begin{corollary} \label{cor2}
For $i \geq 2$, the conditions (a), (b) and (c) of
Lemma~\ref{lemma5} hold with $C$ replaced by $h^i C$ for constants
$0<h<1$ and  $C>0$ depending on $n$ and $M$. In other words,
\begin{equation} \label{aa}
u(\cdot,t_i) \leq (1+ Ch^i \alpha^{1-\e}) \phi(\cdot, t_i)  \hbox{
on } B_{(1-\alpha^\e)r(t_i)}(0),
\end{equation}
  $\Gamma_t(u)$ is contained in the
\begin{equation} \label{aaaa}
 C h^i\alpha r(t) \hbox{-
 neighborhood of }\partial B_{r(t)}(0)
 \end{equation}
  for $t \in [t_i, t_{i+1}]$,
 and
\begin{equation} \label{aaa}
u(\cdot, t_{i}) \leq \phi(\cdot, t_{i}) + C h^i
\alpha^{\frac{\e+1}{2} } r(t_i)
\end{equation}
on $B_{r(t_{i})}(0) - B_{(1-\alpha^\e)r(t_{i})}(0)$.
\end{corollary}
\begin{proof}
As in the proof of Corollary~\ref{cor1}, the conditions (a), (b) and
(c) are satisfies with constant $C=K=C(n,M)\geq 1$ for $i \geq 2$.
Let $m$ be the integer as in Corollary~\ref{cor1} and let $\e=2/3$.
Then for  $t \geq t_{m+4}$, the constants $C$ and $K$ can be
replaced, respectively, by $\beta=\alpha^{\frac{1-\e}{2}}C$ and
$C_1\beta$ (see Corollary~\ref{cor1} and Lemma~\ref{lemma6}). Here
$C_1$ is a constant depending on $n$ and $M$. Then by the condition
(b) with the improved constants, for $ t \geq t_{m+4}$,
\begin{equation} \label{123}
 \Gamma_t(u) \subset
B_{(1+ C_1 \beta \alpha)r(t)}(0) -B_{r(t)}(0).
\end{equation}
for a constant $C_1>0$ depending on $n$ and $M$. Fix $i \geq m+4$.
Decompose $[t_i, t_{i+1}]$ into subintervals of length $\beta \alpha
r^{in}(t_i)^2$ and let $\tau$, $\tilde{\tau}$ and $\Sigma$ be given
similarly as in Lemma~\ref{lemma2}, so that $\Sigma =
B_{r^{in}(\tau)}(0) \times [\tilde{\tau}, \tau]$,
 $\tau -\tilde{\tau} = \beta\alpha r^{in}(t_i)^2$ and $V_{[\tilde{\tau}, \tau]} \leq C(n,M)/r^{in}(t_i)$.
Recall that $V_{[\tilde{\tau}, \tau]}$ is the average velocity of
$\partial \Omega^{in}$ on $[\tilde{\tau}, \tau]$. Then using the
upper bound on $V_{[\tilde{\tau}, \tau]}$,
\begin{equation} \label{12345}
r^{in}(t) -r^{in}(\tau) \leq   r^{in}(\tilde{\tau}) -r^{in}(\tau)
\leq C(n,M)\beta \alpha
 r^{in}(t_i).
\end{equation}
for all $t \in [\tilde{\tau}, \tau]$.  By (\ref{123}) and
(\ref{12345}) with $|\nabla u| \leq C_0M$,
$$
\max_{\partial B_{r^{in}(\tau)}(0)} u(\cdot, t) \leq C(n,M) \beta
\alpha r^{in}(t_i)
$$
for all $t \in [\tilde{\tau}, \tau]$. In other words,
\begin{equation} \label{1235}
u \leq C(n,M) \beta \alpha r^{in}(t_i) \hbox{ on } \partial \Sigma
\end{equation}
Let $\psi$ be a caloric function in $\Sigma$ constructed as in
Lemma~\ref{lemma2}, then by (\ref{1235}) and the improved condition
(c),
\begin{equation}\label{1234}
\psi(\cdot, \tau) \geq u(\cdot, \tau ) -C(n,M) \beta \alpha
r^{in}(t_i) \geq (1-C(n,M)\sqrt{ \beta \alpha})u(\cdot, \tau)
\end{equation}
 on $\partial B_{(1 - c \sqrt{\beta \alpha}) r^{in}(t_i)}(0)$.
Using (\ref{1234}) instead of (\ref{12}),  the construction of
$\Omega_1$  can be improved so that $\partial \Omega_1 \cap \{t_i
\leq t \leq t_{i+1}\}$ is located in the $C(n,M)\beta \alpha
r^{in}(t_i)$-neighborhood of $\partial \Omega^{in}$, for $i \geq
m+4$. Then using the bound $|\nabla u| \leq C_0M$ again,
$$\max_{\partial B_{r(t)}(0)}u(\cdot, t) \leq C(n,M) \beta \alpha
r(t).$$  Note that the above  inequality gives (\ref{newnewnew}),
(\ref{new}) and (\ref{betabeta}) for $K=C_1\beta<1$ and $t\geq
t_{m+4}$. Then as mentioned in Remark 6 and Remark 7,  we
 iterate Lemma~\ref{lemma5} and Lemma~\ref{lemma6} for $K <1$,
  improving  the  approximating region $\Omega_1$ at later times.
\end{proof}

\section{Asymptotic behavior of $u$; Regularity of $\Gamma(u)$ }
\label{sec8}

 (\ref{aaaa})
of  Corollary~\ref{cor2} says that the free boundary of $u$ is
asymptotically spherical. Using this result, we approximate $u$ by
radially symmetric functions $w_k$ supported on $\Omega_1 \cap \{t_k
\leq t <T\}$  such that  $w_k$ is caloric   and its gradient is
close to $1$ on $\partial \Omega_1$. Then $u$ turns out to be
asymptotically self-similar  by Lemma~\ref{ghv}, and
 we also obtain the  regularity of $\Gamma(u)$ by Lemma~\ref{AW}.

\begin{proposition} \label{lastlemma} [(i)and (ii) of Theorem~\ref{thm1}]
 Let $0<h=h(n,M)<1$  be  as in
Lemma~\ref{lemma5} and let $\e=2/3$. Then for $k \geq 2$, there
exists a radially symmetric caloric function $w_k$ defined in
$\Omega_1 \cap \{t_k \leq t <T\}$ such that

\begin{itemize}
\item[(i)]  For $t \geq t_k$, $\Gamma_t(u)$ is located in the $C h^k\alpha r(t)$-neighborhood of
$\Gamma_t(w_k)$
\item[(ii)] For $t \geq t_k$
\begin{equation} \label{firstin}
w_k(\cdot,t) \leq u(\cdot,t) \leq w_k(\cdot,t) +C h^k\alpha^{1-\e} \max u(\cdot, t)
\end{equation}
 where we let $w_k = 0$ outside
$\Omega(w_k)$
\item[(iii)] On $\partial \Omega(w_k)$
\begin{equation} \label{secondin}
1-C h^{Ak} \alpha^{A\e}  \leq |\nabla w_k| \leq 1
\end{equation}
where $A=A(n,M)>0$ is a constant depending on $n$ and $M$.
\end{itemize}
In (i), (ii) and (iii) $C$ denotes a constant depending on $n$ and
$M$. By (i), the free boundary $\Gamma_t(u)$ is asymptotically
spherical and by Lemma~\ref{ghv} with (i), (ii) and (iii),  $u$ is
asymptotically self-similar.
\end{proposition}

\noindent {\bf Remark 8.}
(ii) of Theorem~\ref{thm1} follows from (\ref{firstin}), (\ref{self2}) and Lemma~\ref{lemma1}.

\begin{proof}
 Recall that $u$ is well approximated by a radial function
$\phi$, which is caloric on each time interval $(t_i, t_{i+1})$.
However $\phi$ does not solve a heat equation on $(t_k,T)$ since it
is discontinuous at each $t_i$ ($\phi(\cdot, t_i)$ is defined to be
the maximal radial function $\leq u(\cdot,t_i)$). Hence we construct
another radial function $w_k \leq \phi$ which is caloric on $(t_k,
T)$. Using Corollary~\ref{cor2}, we  show that the values of $w_k$
are close to the values of $u$ and the gradient of $w_k$ is
sufficiently close to $1$ on its vanishing boundary.

For $k \geq 2$, let $w_k$ solve
$$
\left\{\begin{array} {lll} \Delta w_k= \partial w_k/\partial t
&\hbox{ in }& \Omega_1 \cap \{t > t_{k-1}\}
\\ \\
w_k= \phi &\hbox{ on }& \{t=t_{k-1}\}
\\ \\
 w_k= 0 &\hbox{ on  }& \partial \Omega_1 \cap \{t >t_{k-1}\}
\end{array}\right.
$$
and let $\tilde{w}_k$ solve
$$
\left\{\begin{array} {lll} \Delta \tilde{w}_k= \partial
\tilde{w}_k/\partial t &\hbox{ in }& \Omega_1 \cap \{t > t_{k-1}\}
\\ \\
\tilde{w}_k= u &\hbox{ on }& \{t=t_{k-1}\}
\\ \\
  \tilde{w}_k= 0 &\hbox{ on  }& \partial \Omega_1 \cap \{t >t_{k-1}\}.
\end{array}\right.
$$
Then for  $ t\geq t_k$,
\begin{equation} \label{22}
 w_k \leq \tilde{w}_k \leq
(1+Ch^k\alpha^{1-\e}) w_k
\end{equation}
where $C=C(n,M)>0$ and the second inequality follows from (\ref{aa})
and (\ref{aaa}) with $\e=2/3$.  For $i\geq k-1$, let $v_i$ be a
caloric function defined in $\Omega_1 \cap \{t > t_{k-1}\}$ with the
following boundary condition
$$
\left\{\begin{array} {lll} v_i= 0 &\hbox{ on }& \{t=t_{k-1}\}
\\  \\
v_i= u &\hbox{ on  }& \partial \Omega_1 \cap \{t_i < t < t_{i+1}\}
\\ \\
 v_i= 0 &\hbox{ on  }& \partial \Omega_1 \cap \{t_{k-1}<t<t_i \hbox{ or } t>t_{i+1}\}
.
\end{array}\right.
$$
Then  in $\Omega_1 \cap \{t > t_{k-1}\}$
\begin{equation} \label{11}
u = \tilde{w}_k+\sum_{i= k-1}^{\infty} v_i.
\end{equation}
Throughout the proof, let $C$ denote a positive constant depending
on $n$ and $M$. Then by (\ref{aaaa}) with $|\nabla u| \leq C_0M$
 (Lemma~\ref{cak}),
\begin{equation} \label{33}
v_i=u \leq C  h^i \alpha  r(t_i)
\end{equation}
on $\partial \Omega_1 \cap \{ t_i < t < t_{i+1}\}$.
 Hence in  $ \Omega_1 \cap \{t \geq t_{i+2}  \}$,
\begin{equation} \label{44}
v_i \leq C   h^i \alpha u.
\end{equation}
Combining (\ref{22}), (\ref{11}),  (\ref{33}) and   (\ref{44}), we
obtain that in $\Omega_1 \cap \{t \geq t_{k} \}$
\begin{eqnarray}
 u &\leq& (1+Ch^k \alpha^{1-\e})w_k + \sum_{i=k-1}^{\infty}C
h^i \alpha u +
C  h^k \alpha \max u(\cdot,t)  \nonumber\\
& \leq& (1+Ch^k \alpha^{1-\e})w_k + C  h^k \alpha \max  u(\cdot, t)
\label{cha}  \\
& \leq& w_k + C  h^k \alpha^{1-\e} \max u(\cdot, t). \nonumber
\end{eqnarray}
Also for $t \in (t_i, t_{i+1})$, $i \geq k$, and $x \in \Omega_t(u)
- \Omega_t(w_k)$,
$$u(x,t) \leq C  h^i \alpha  r(t_i)
\leq C  h^k \alpha  r(t_i)
 \leq C  h^k \alpha^{1-\e} \max u(\cdot, t)$$ where the first
inequality follows from a similar argument as in (\ref{33}). Hence
we obtain the second part (ii) of the lemma. Observe that the first
part (i) follows from Corollary~\ref{cor2} since $\Gamma_t(w_k) =
\partial B_{r(t)}(0)$ for $t \geq t_k$.

Next we prove (iii) that $|\nabla w_k|$ is sufficiently close to $1$
on $\partial \Omega_1 \cap \{t \geq t_{k}\}$. Since $w_k \leq u$ and
the free boundary $\Gamma_t(w_k)$, that is $\partial B_{r(t)}(0)$,
intersects $\Gamma_t(u)$ at each $t$, we obtain the upper bound
$$|\nabla w_k| \leq 1 \hbox{ on } \partial \Omega_1 \cap \{t \geq t_{k}\}.$$

To obtain the lower bound of $|\nabla w_k|$, i.e., for the first
inequality of (\ref{secondin}),   we compare $w_k$ with some
harmonic functions near the vanishing boundary $\partial
B_{r(t)}(0)$. Fix a dyadic interval $(t_i, t_{i+1}] \subset (t_{k},
T)$. For $t \in (t_i, t_{i+1}] $, let $H_{(t)}(\cdot)$ be the
harmonic function defined in $B_{r(t)}(0) -
B_{(1-h^{k/2}\alpha^{\e})r(t)}(0)$ with the following boundary data
$$
\left\{\begin{array} {lll} H_{(t)}= 0 &\hbox{ on }&\partial
B_{r(t)}(0)
\\ \\
 H_{(t)}(\cdot)= w_k(\cdot, t) &\hbox{ on  }& \partial
B_{(1-h^{k/2}\alpha^{\e})r(t)}(0).
\end{array}\right.
$$
Then by Lemma~\ref{ca} applied to $w_k$,
 \begin{equation} \label{H12}
 H_{(t)}(\cdot) \leq
(1+h^{ak/2}\alpha^{a\e} )w_k(\cdot,t)
\end{equation}
on $B_{r(t)}(0) - B_{(1-h^{k/2}\alpha^{\e})r(t)}(0)$ where $a>0$ is
a constant depending on $n$ and $M$.
 This implies
\begin{equation} \label{hw1}
|\nabla H_{(t)}| \leq (1+h^{ak/2}\alpha^{a\e} )|\nabla w_k|
\end{equation}
 on $\partial B_{r(t)}(0)$. Let $$A=A(n,M) = \min \{a/2, 1/2\}>0.$$ Then by (\ref{hw1})
 it suffices to prove
 \begin{equation} \label{H_t}
|\nabla H_{(t)}| \geq 1 - Ch^{Ak}\alpha^{A\e}
\end{equation}
on $\partial B_{r(t)}(0)$ for $t \in (t_i, t_{i+1}] \subset (t_k,
T)$.

First we show (\ref{H_t}) for time $t$ in some  subset $\{s_1, ...,
s_m\}$ of the interval $(t_i, t_{i+1}]$. Recall
\begin{itemize}
\item[(a-1)]  By Lemma~\ref{lemma2}, the inner-radius $r(t)$ is  Lipschitz  on
$[t_{i-1}, t_{i+1}]$, i.e.,
$$
|r(t)-r(s)| \leq C|t-s|/r(t_i)
$$
for $t, s \in [t_{i-1}, t_{i+1}]$
\item[(a-2)] By Corollary~\ref{cor2}, the outer-radius $r^{out}(t)$ satisfies
$$r(t) \leq r^{out}(t)\leq r(t) + Ch^i \alpha r(t_i) \leq r(t) + Ch^k \alpha r(t_i)$$
for $t \in [t_{i-1}, t_{i+1}]$.
\end{itemize}
Also recall that  $r^{out}(t)$ is not necessarily Lipschitz on
$[t_{i-1}, t_{i+1}]$. However using the properties (a-1) and (a-2),
we can construct a Lipschitz function $R(t)$ on $[t_{i-1}, t_{i+1}]$
such that
\begin{itemize}
\item[(b-1)]
$ |R(t)-R(s)| \leq C|t-s|/r(t_i) $ for $t, s \in [t_{i-1}, t_{i+1}]$
\item[(b-2)]$r^{out}(t) \leq R(t) \leq r(t)+ C h^k \alpha r(t_i)$
\item[(b-3)] $R(t) = r^{out}(t)$ for $t $ in some subset $ \{s_1,  ..., s_m\}$
 of $[t_i, t_{i+1}]$ such that\\
 $t_i =s_0<s_1< ...< s_m<s_{m+1}=t_{i+1}$ and
$$
 s_j-s_{j-1} \leq h^k \alpha r^2(t_i)
$$
 for  $1\leq j \leq m+1$.
\end{itemize}
Now let $\tilde{\Omega}$ be a space time region  on the time
interval $[t_{i-1}, t_{i+1}]$ such that
$$\tilde{\Omega}_t  = B_{R(t)}(0)$$
for $t \in [t_{i-1}, t_{i+1}]$. Let $\tilde{u}$ solve
$$
\left\{\begin{array} {lll} \Delta \tilde{u}= \tilde{u}_t &\hbox{ in
}& \tilde{\Omega}
\\ \\
\tilde{u}= u &\hbox{ on }& \{t=t_{i-1}\}
\\ \\
 \tilde{u}= 0 &\hbox{ on  }& \partial \tilde{\Omega} \cap \{t_{i-1} <t <t_{i+1}\}.
\end{array}\right.
$$
Then by the construction of $R(t)$,
\begin{equation} \label{x0}
u \leq \tilde{u} \leq u +C h^k \alpha r(t_i)
\end{equation}
where the first inequality follows from the first inequality of
(b-2) and the last inequality follows from the last inequality of
(b-2) with $|\nabla u| \leq C_0M$.

Fix $t \in \{s_1,...,s_m\}$. Then $\Gamma_t(u)$ intersects
$\Gamma_t(\tilde{u})$ since $\Gamma_t(\tilde{u})= \partial
B_{r^{out}(t)}(0)$.  Let $x_0 \in \Gamma_t(u) \cap
\Gamma_t(\tilde{u})$, then by (\ref{x0})
\begin{equation} \label{x00}
|\nabla \tilde{u}(x_0, t)| \geq 1.
\end{equation}
On the other hand, let $\tilde{H}(\cdot)$ be the harmonic function
defined in $B_{r^{out}(t)}(0) - B_{(1-h^{k/2}\alpha^{\e})r(t)}(0) $
with the following boundary data
$$
\left\{\begin{array} {lll} \tilde{H}= 0 &\hbox{ on }&\partial
B_{r^{out}(t)}(0)
\\  \\
\tilde{H}= m &\hbox{ on  }& \partial
B_{(1-h^{k/2}\alpha^{\e})r(t)}(0)
\end{array}\right.
$$
where
$$
m := \max\{ u(x, t):x \in\partial
B_{(1-h^{k/2}\alpha^{\e})r(t)}(0)\} + C h^k \alpha r(t_i).
$$
Then by Lemma~\ref{ca} applied to $\tilde{u}$ with (\ref{x0}),
\begin{equation} \label{H123}
\tilde{H}(\cdot) \geq (1-h^{ak/2}\alpha^{a\e} )\tilde{u}(\cdot, t)
\end{equation}
in $B_{r^{out}(t)}(0) - B_{(1-h^{k/2}\alpha^{\e})r(t)}(0) $ where
$a=a(n,M)>0$. Hence on $\partial B_{r^{out}(t)}(0)$,
\begin{equation}\label{last}
|\nabla \tilde{H}| \geq (1-h^{ak/2}\alpha^{a\e} )|\nabla
\tilde{u}(x_0, t)|  \geq 1-h^{ak/2}\alpha^{a\e}
\end{equation}
where the last inequality follows from (\ref{x00}).

Now we compare the harmonic functions $H_{(t)}$ and $\tilde{H}$ by
comparing their boundary values  $w_k(\cdot,t)$ and $m$ on $\partial
B_{(1-h^{k/2}\alpha^{\e})r(t)}(0)$. (Recall $t\in \{s_1,...,s_m\}$
is
 fixed.) For  $x \in \partial
B_{(1-h^{k/2}\alpha^{\e})r(t)}(0)$
\begin{eqnarray}
m
 & \leq & (1+Ch^k \alpha^{1-\e})w_k(x,t) + C  h^k \alpha r(t_i)  \nonumber \\
& \leq & (1+Ch^{k/2} \alpha^{1-\e})w_k(x,t) \label{mw1}
\end{eqnarray}
where the first inequality follows from (\ref{cha}) with the
construction of $m$ and last inequality follows since (\ref{cha})
and the almost harmonicity of $w_k$  imply that $w_k \approx
h^{k/2}\alpha^{\e}r(t) $ on $\partial
B_{(1-h^{k/2}\alpha^{\e})r(t)}(0)$.
 Then on $\partial B_{r(t)}(0)$
\begin{eqnarray}
|\nabla H_{(t)}|  &\geq& (1-C h^{k/2} \alpha^{1-\e})\frac{
w_k|_{\partial B_{(1-h^{k/2}\alpha^{\e})r(t)}(0)} }{m}|\nabla
\tilde{H}| \nonumber \\&\geq& 1-Ch^{Ak}\alpha^{A\e} \label{s_j}
\end{eqnarray}
 where  the first inequality follows from
 the constructions of $H_{(t)}$ and $\tilde{H}$ with (b-2) and
 the last inequality follows from  (\ref{last}) and (\ref{mw1}) with
the  constants
 $A= \min \{a/2, 1/2\}$ and $\e=2/3$.
Hence we obtain the desired inequality (\ref{H_t}) for time $t $ in
the subset $\{s_1, ..., s_m\}$ of $(t_i, t_{i+1}]$.

Next we show (\ref{H_t}) for $t \in (s_{j-1}, s_j)$, $1 \leq j \leq
m$. Since $\phi$ is decreasing in time on each dyadic time interval
and the region $\Omega_1$ is shrinking in time, $w_k$ is also
decreasing in time. Hence on $\partial
B_{(1-h^{k/2}\alpha^{\e})r(s_j)}(0)$
\begin{equation} \label{with1}
w_k(\cdot, t) \geq w_k(\cdot, s_j) .
\end{equation}
By (a-1) with $|t-s_j|\leq s_j-s_{j-1} \leq h^k \alpha r^2(t_i)$,
\begin{equation} \label{with2}
0 \leq r(t)-r(s_j) \leq C h^k \alpha r(t).
\end{equation}
 Then by (\ref{with1}) and (\ref{with2}) with the
almost harmonicity of $w_k$,
\begin{equation} \label{with3}
w_k(\cdot, t)|_{\partial B_{(1-h^{k/2}\alpha^{\e})r(t)}(0)} \geq
(1-C h^{k/2} \alpha^{1-\e}) w_k(\cdot, s_j)|_{\partial
B_{(1-h^{k/2}\alpha^{\e})r(s_j)}(0)}.
\end{equation}
Hence we obtain
\begin{eqnarray*}
|\nabla H_{(t)}|_{\partial   B_{r(t)}(0)} &\geq& (1- Ch^{k/2}
\alpha^{1-\e}) |\nabla H_{(s_j)}|_{\partial   B_{r(s_j)}(0)} \\
&\geq& 1-C h^{Ak} \alpha^{A\e}
\end{eqnarray*}
where the first inequality follows from the construction of
$H_{(t)}$ with (\ref{with2}) and (\ref{with3}),  the last inequality
follows from (\ref{s_j}).  Since $t \in (s_{j-1}, s_j]$ for $1 \leq
j \leq m$, we can conclude that (\ref{H_t}) holds for all $t \in
(s_0, s_m] =(t_i, s_m] \subset (t_i, t_{i+1}]$. Then by repeating
the above argument with $t_{i+1}$ replaced by $t_{i+2}$, we can
obtain (\ref{H_t}) for all $t \in (t_i, t_{i+1}]$. Recall that
(\ref{H_t}) implies the first inequality of (\ref{secondin}). Hence
we obtain the properties (i), (ii) and (iii) of the proposition
 for the radial function $w_k$ for $k \geq 2$.

Observe that by (i) and (ii),
\begin{equation} \label{self1}
\sup_{t_k < t< T} \|u(\cdot, t)-w_k(\cdot,t)\|_\infty/\| u(\cdot,
t)\|_\infty \rightarrow 0
\end{equation}
and
\begin{equation} \label{self11}
\sup_{t_k < t< T}{\rm dist}(\Gamma_t (u), \Gamma_t (w_k)) / r(t)
\rightarrow 0
\end{equation}
 as $k
\rightarrow \infty$ where $r(t) = $ diameter of $\Gamma_t (w_k)$/2.
On the other hand, (iii) implies that the radial function $w_k$ is a
supersolution of ($P$) and also the function
$(1+Ch^{Ak}\alpha^{A\e})w_k$ is a subsolution of ($P$), both of
which vanish at time $t=T$. Hence for some  constant $1 \leq \beta
\leq 1+Ch^{Ak}\alpha^{A\e}$, a radial solution $v$ of ($P$)
 vanishes at time $t=T$ if $v$ has an initial condition
 $v(\cdot,t_{k-1}) = \beta w_k(\cdot,t_{k-1}) = \beta \phi(\cdot, t_{k-1})$.
Then by a similar argument as in the proof of Lemma~\ref{lemma1}, we
can show that the free boundary $\Gamma_t(v)$ is located in the
$C(n,M)h^{Ak}\alpha^{A\e}r(t)$-neighborhood of $\Gamma_t(w_k)$ since
$v$ and $w_k$, otherwise, would have different extinction times.
Then using the upper bounds of $|\nabla w_k|$ and $|\nabla v|$
(Lemma~\ref{cak}),
\begin{equation} \label{self2}
 |v(\cdot, t)-w_k(\cdot, t)| \leq
C(n,M)h^{Ak}\alpha^{A\e}r(t)
\end{equation}
where $r(t) \approx \|w_k(\cdot, t)\|_\infty$. By Lemma~\ref{ghv},
$v$ is asymptotically self-similar and hence we can conclude from
(\ref{self1}), (\ref{self11}) and (\ref{self2}) that $u$ is
asymptotically self-similar.
\end{proof}

The next corollary  follows from Lemma~\ref{AW} and the flatness of
 $\Gamma(u)$. Note that it was proved in  [W] that a limit solution of ($P$) is also a
solution in the sense of domain variation.

\begin{corollary} \label{cor} [(iii) of Theorem~\ref{thm1}]
Let $0<h=h(n,M)<1$  be  as in Lemma~\ref{lemma5}. Then there is a constant $c_0>0$ depending only on $n$ and $M$ such that  if $k\geq 2$ and 
$$\alpha h^k < c_0$$
 then $\Gamma(u) \cap
\{t_k<t<T\}$ is a graph of $C^{1+\gamma, \gamma}$ function and the
space normal is H\"{o}lder continuous. In fact,  we can take 
$$c_0 = \frac{C_nC_3}{CC_4} \min\{\frac{1}{L}, \frac{1}{M^2} \}$$
where $C_n$ is a dimensional constant, $C$ is a constant given as in (i) of proposition~\ref{lastlemma}, $C_3$ and $C_4$ are constants given as in (\ref{rt}) of Lemma~\ref{lemma1}, $L=L(n,M)$ is the Lipschitz constant for $\Omega_1$, which is  given as in Lemma~\ref{lemma2}.
\end{corollary}

\begin{proof}
Let
 $$(y,\tau) \in \Gamma(u) \cap \{t_{k}<t<T\}$$ where $k$ will be
 chosen later in the proof. We will obtain the regularity of $\Gamma(u)$ in a parabolic cube containing $(y , \tau)$, using Lemma~\ref{AW} and Proposition~\ref{lastlemma}. Without loss of generality, we assume that
 $y=(0,...,0,y_n)$ with $y_n >0$, and $\tau \in (t_k, t_{k+1}]$. Let  $0<\sigma_1<1$ be the constant given  as in Lemma~\ref{AW}, and let 
 $$\rho = \delta r(\tau)$$ where $\delta>0$ is a constant
  depending on  $n$ and  $M$, which will be chosen later.

Let $w_k$ be the caloric function constructed as in the proof of
Proposition~\ref{lastlemma}. Recall that $\Gamma(w_k)=\partial
\Omega_1\cap \{t_{k-1}<t<T\} $ is Lipschiz in a parabolic scaling
(see Lemma~\ref{lemma2}). Then by (i) of Proposition~\ref{lastlemma} and
the Lipschitz property of $\Gamma(w_k)$, 
\begin{equation} \label{AW1}
u=0 \,\,\hbox{ in } \,\,Q^-_\rho(y, \tau) \cap \{x: x_n \geq  y_n
+\sigma_1 \rho\} 
\end{equation} 
if 
\begin{equation} \label{AW111} C h^k \alpha < \sigma_1 \delta \,\, \hbox{ and } \,\, \frac{L\delta^2r(\tau)^2}{r(\tau)} \leq \sigma_1 \delta r(\tau)
\end{equation}
where $C =C(n,M)>0$ and $0<h=h(n,M)<1$ are constants given as in (i) of Proposition~\ref{lastlemma}, and $L=L(n,M)>0$ is the Lipschitz constant for $\Gamma(w_k)$ in $Q^-_{r(\tau)}(y, \tau)$. 
Here observe that  (\ref{AW111}) holds if
 $\delta$ and $k$ satisfy
$$
 \delta \leq \sigma_1 / L  \,\, \hbox{ and }\,\,  Ch^k \alpha < \sigma_1 \delta.
 $$
For (\ref{AW1}) as well as for   the arguments below, 
 we choose  $\delta = \delta(n,M)>0$ and $k=k(n,M) \in \N$  satisfying   
  \begin{equation} \label{hsize}
  \delta = \min \{ \frac{\sigma_1}{L}, \frac{c_n \sigma_1^3}{(C_0 M)^2} \}  \,\, \hbox{ and }\,\,  Ch^k \alpha < (\frac{C_3}{2C_4})\sigma_1 \delta < \sigma_1 \delta
  \end{equation}
 where $c_n>0$ is a dimensional constant, $C_0$ is a constant given as in Lemma~\ref{cak}, and   $C_3$, $C_4$ are constants given as in (\ref{rt}).

Next we show $$|\nabla u| \leq 1 +\sigma_1^3 \,\,\hbox{ in
}\,\,Q^-_\rho(y, \tau).$$
 Let
$\tilde{\Omega} \subset \R^n \times [t_{k-1}, t_{k+1}]$ be the
Lipschitz region constructed as in the proof of
Proposition~\ref{lastlemma}, which contains $\Omega(u)\cap \{t_{k-1}
\leq t \leq t_{k+1}\}$. Then since $\max\{|\nabla u|^2-1,0\}$  is
subcaloric  in $\Omega(u)\cap \{t_{k-1}<t<t_{k+1}\} $,
\begin{equation} \label{1leqv}
 |\nabla u|^2-1 \leq v
 \end{equation}
where $v$ solves
$$
\left\{\begin{array} {lll}
v_t=\Delta v  &\hbox{ in } & \tilde{\Omega} \\ \\
v= 0  &\hbox{ on } & \partial \tilde{\Omega} \cap \{t_{k-1}<t<t_{k+1}\} \\ \\
v = \max\{|\nabla u|^2-1,0\} &\hbox{ on } & \{t= t_{k-1}\}.
\end{array}\right.
$$
Observe that the initial condition $|\nabla u_0| \leq M$ implies $|\nabla u| \leq C_0 M$  by Lemma~\ref{cak}, and hence  
 $$\max_{\tilde{\Omega}} v \leq
(C_0M)^2.$$
  Also by Lemma~\ref{ca}, $v(\cdot, t)$ is almost
harmonic near its vanishing boundary  $\partial \tilde{\Omega}_t$
for $t \in [(t_{k-1}+t_k)/2, t_{k+1}]$. Hence for $t
\in [(t_{k-1}+t_k)/2, t_{k+1}]$  and $x \in \tilde{\Omega}$ with ${\rm dist}(x, \partial \tilde{\Omega}_t) \leq
 3\rho = 3\delta r(\tau)$, we get 

\begin{equation} \label{AW2}
 v(x,t) \leq \sigma_1^3 
  \end{equation}
if  
\begin{equation} \label{deltabound}
\delta \leq  c_n \sigma_1^3/\max_{\tilde{\Omega}} v  =c_n \sigma_1^3/(C_0M)^2
\end{equation} where $c_n>0$ is a dimensional constant. Observe that (\ref{deltabound}) follows from (\ref{hsize}).

 On the other hand,  from the construction of $\tilde{\Omega}$,  its
boundary  $\partial \tilde{\Omega}_t$  is located in the $Ch^k\alpha
r(t_{k})$-neighborhood of $\Gamma_t(u)$ for $t \in [t_{k-1},
t_{k+1}]$. Hence for $(x,t) \in \Gamma(u) \cap  Q^-_\rho(y, \tau)$,
\begin{equation} \label{AW3}
 {\rm dist}(x,
\partial \tilde{\Omega}_t) \leq Ch^k\alpha r(t_k) \leq (\frac{C_3}{2C_4})\sigma_1 \delta r(t_k) 
\leq \sigma_1  \delta r(\tau) \leq \delta r(\tau) =\rho
\end{equation}
where the second inequality follows from 
 (\ref{hsize}) and the third inequality follows from (\ref{rt}) since $\tau \in (t_k , t_{k+1}]$.  Then by
(\ref{AW2}) and (\ref{AW3}),
\begin{equation} \label{AW4}
|\nabla u|^2 -1 \leq v \leq \sigma_1^3  \,\,\hbox{ in }\,\,
Q^-_\rho(y, \tau)
\end{equation}
where the first inequality follows from (\ref{1leqv}).

By Lemma~\ref{AW} with (\ref{AW1}) and (\ref{AW4}), we conclude that
$\Gamma(u)\cap \{t_k <t<T\}$ is a graph of $C^{1+\gamma, \gamma}$
function and the space normal is H\"{o}lder continuous.

\end{proof}


\begin{thebibliography}{[MT]}
\bibitem[ACS]{acs1} I. Athanasopoulos, L. Caffarelli and S. Salsa, { Caloric functions
in Lipschitz domains and the Regularity of Solutions to Phase
Transition Problems,} {\em{Ann. of Math. }}{\bf 143 (2)} (1996),
413-434.
\bibitem[AW]{aw} J. Andersson and G. Weiss, { A parabolic free boundary problem with
Bernoulli type condition on the free boundary, }  {\em
 J. Reine Angew. Math.}, {\bf 627 }(2009), 213-235.

\bibitem[BHL]{bhl} C. Brauner, J. Hulshof and A. Lunardi, { A
grneral approach to stability in free boundary problems, } {\em{J.
Differential Equations }}{\bf 164 } (2000),  16-48.

\bibitem[BHS]{bhs} C. Brauner, J. Hulshof and C. Schmidt-Laine, { The saddle
point property for focusing selfsimilar solutions in a free boundary
problem, } {\em{Proc. Amer. Math. Soc.}} {\bf 127, (2) } (1999),
473-479.

\bibitem[BL]{bl} J. Buckmaster and G. Ludford , { Theory of Laminar Flames, } {\em Cambridge University Press.}
Cambridge, 1982.


\bibitem[CLW1]{CLW1} L. Caffarelli, C. Lederman and N. Wolanski, { Uniform estimates and limits for a two phase parabolic
singular perturbation problem.} {\em { Indiana Univ. Math. J.}}{\bf 46(2)} (1997), 453-489.

\bibitem[CLW2]{CLW2} L. Caffarelli, C. Lederman and N. Wolanski, {Pointwise and viscosity solutions for the limit of a two
phase parabolic singular perturbation problem.} {\em { Indiana Univ. Math. J.}} {\bf 46(3)} (1997), 719-740.


\bibitem[CV]{cv} L. Caffarelli and J. V\'{a}zquez, { A free-boundary problem
for the heat equation arising in flame propagation, } {\em Trans.
Amer. Math. Soc.} {\bf 347 (2)} (1995), 411-441.

\bibitem[D] {d} E. DiBenedetto, { Partial Differential Equations, } {\em Birkh\"{a}user Boston, Inc.}
Boston, (1995).

\bibitem[GHV]{ghv} V. Galaktionov, J. Hulshof and J. V\'{a}zquez, { Extinction and focusing
behaviour of spherical and annular flames described by a free
boundary problem, } {\em J. Math. Pures Appl.} {\bf 76} (1997),
563-608.

\bibitem [HH] {HH} D. Hilhorst and J. Hulshof, {An elliptic-parabolic problem in combustion theory: convergence
to travelling waves,} {\em Nonlinear Anal.}, {\bf 17} (1991), 519-546.

\bibitem[K]{k} I. Kim, { A free boundary problem
arising in flame propagation, } {\em J. Differential Equations,}
{\bf 191} (2003), 470-489.

\bibitem [LVW]  {lvw} C. Lederman, J. V\'{a}zquez and N. Wolanski, 
{ Uniqueness of solution to a free boundary problem from combustion,} { \em Trans. Am. Math. Soc.} {\bf 353(2)} (2001), 655-692.

\bibitem[LW] {LW} C. Lederman and N. Wolanski,{ Viscosity solutions and regularity of the free boundary
for the limit of an elliptic two phase singular perturbation problem,} {\em Annali Scuola Norm.
Sup. Pisa Cl. Sci. }, {\bf (4) 27} (1998), 253-288.

\bibitem[P] {P} A. Petrosyan, {On existence and uniqueness in a free boundary problem from combustion,} { \em Comm.
Partial Differential  Equations,} {\bf 27(3-4)} (2002), 763-789.

\bibitem [V] {V} J. V\'{a}zquez, 
{ The free boundary problem for the heat equation with fixed gradient condition,}  {\em 
Free Boundary Problems, Theory and Applications, Zakopane, 1995, Pitman Res. Notes Math. Ser.,} {\bf 
vol. 363} (1996),  277-302.

\bibitem[W]{s} G. Weiss, { A Singular limit arising in combustion theory: Fine properties of the free boundary, }
{\em Calc. Var.} {\bf 17} (2003), 311-340.
\end{thebibliography}
\end{document}